\documentclass[11pt]{article}

\usepackage{fullpage}
\usepackage{subcaption}

\usepackage{enumerate}
\usepackage{amsmath,amsfonts,amssymb,amsthm, bbold}

\usepackage{graphics,esint}
\usepackage{epsfig}
\usepackage{xcolor}
\usepackage{xspace}
\definecolor{pingreen}{rgb}{0,39,14}

\usepackage{xfrac}
\usepackage{comment}
\usepackage{hyperref}
\usepackage{cleveref}

\crefname{section}{§}{§§}
\Crefname{section}{§}{§§}

\makeatletter
\newtheorem*{rep@theorem}{\rep@title}
\newcommand{\newreptheorem}[2]{%
\newenvironment{rep#1}[1]{%
 \def\rep@title{#2 \ref{##1}}%
 \begin{rep@theorem}}%
 {\end{rep@theorem}}}
\makeatother

\newtheorem{theorem}{Theorem}[section]

\newtheorem{lemma}[theorem]{Lemma}

\newtheorem{remark}[theorem]{Remark}

\newtheorem{corollary}[theorem]{Corollary}

\newreptheorem{theorem}{Theorem}

\newcommand{\R}{\mathbb{R}}
\newcommand{\Z}{\mathbb{Z}}

\newcommand{\Fcal}{\mathcal{F}}

\def\T{\mathbb{T}}

\def\N{\mathbb N}
\def\eps{\varepsilon}

\def\eps{\varepsilon}
\def\loc{\mathrm{loc}}

\def\d {\,\mathrm {d}}
\def\dx{\,\mathrm {d}x}
\def\dz{\,\mathrm {d}z}

\def\dt{\,\mathrm {d}t}
\def\dy{\,\mathrm {d}y}

\def\lin{\mathrm{lin}}

\def\de0#1{\rule[3pt]{#1}{0.4pt} \hspace{-0.1pt} \rule[3.05pt]{0.05pt}{0.4pt} \hspace{-0.1pt} \rule[3.1pt]{0.05pt}{0.4pt} \hspace{-0.1pt} \rule[3.15pt]{0.05pt}{0.4pt} \hspace{-0.1pt} \rule[3.2pt]{0.05pt}{0.4pt} \hspace{-0.1pt} \rule[3.25pt]{0.05pt}{0.4pt} \hspace{-0.1pt} \rule[3.3pt]{0.05pt}{0.4pt} \hspace{-0.1pt} \rule[3.35pt]{0.05pt}{0.4pt} \hspace{-0.1pt} \rule[3.4pt]{0.05pt}{0.4pt} \hspace{-0.1pt} \rule[3.45pt]{0.05pt}{0.4pt} \hspace{-0.1pt} \rule[3.5pt]{0.05pt}{0.4pt} \hspace{-0.1pt} \rule[3.55pt]{0.05pt}{0.4pt} \hspace{-0.1pt} \rule[3.6pt]{0.05pt}{0.4pt} \hspace{-0.1pt} \rule[3.65pt]{0.05pt}{0.4pt} \hspace{-0.1pt} \rule[3.7pt]{0.05pt}{0.4pt} \hspace{-0.1pt} \rule[3.75pt]{0.05pt}{0.4pt} \hspace{-0.1pt} \rule[3.8pt]{0.05pt}{0.4pt} \hspace{-0.1pt} \rule[3.85pt]{0.05pt}{0.4pt} \hspace{-0.1pt} \rule[3.9pt]{0.05pt}{0.4pt} \hspace{-0.1pt} \rule[3.95pt]{0.05pt}{0.4pt} \hspace{-0.1pt} \rule[4.0pt]{0.05pt}{0.4pt} \hspace{-0.1pt} \rule[4.05pt]{0.05pt}{0.4pt} \hspace{-0.1pt} \rule[4.1pt]{0.05pt}{0.4pt} \hspace{-0.1pt} \rule[4.15pt]{0.05pt}{0.4pt} \hspace{-0.1pt} \rule[4.2pt]{0.05pt}{0.4pt} \hspace{-0.1pt} \rule[4.25pt]{0.05pt}{0.4pt} \hspace{-0.1pt} \rule[4.3pt]{0.05pt}{0.4pt} \hspace{-0.1pt} \rule[4.35pt]{0.05pt}{0.4pt} \hspace{-0.1pt} \rule[4.4pt]{0.05pt}{0.4pt} \hspace{-0.1pt} \rule[4.45pt]{0.05pt}{0.4pt} \hspace{-0.1pt} \rule[4.5pt]{0.05pt}{0.4pt} \hspace{-0.1pt} \rule[4.55pt]{0.05pt}{0.4pt} \hspace{-0.1pt} \rule[4.6pt]{0.05pt}{0.4pt} \hspace{-0.1pt} \rule[4.65pt]{0.05pt}{0.4pt} \hspace{-0.1pt} \rule[4.7pt]{0.05pt}{0.4pt} \hspace{-0.1pt} \rule[4.75pt]{0.05pt}{0.4pt} \hspace{-0.1pt} \rule[4.8pt]{0.05pt}{0.4pt} \hspace{-0.1pt} \rule[4.85pt]{0.05pt}{0.4pt} \hspace{-0.1pt} \rule[4.9pt]{0.05pt}{0.4pt} \hspace{-0.1pt} \rule[4.95pt]{0.05pt}{0.4pt} \hspace{-0.1pt} \rule[5.0pt]{0.05pt}{0.4pt} \hspace{-0.1pt} \rule[5.05pt]{0.05pt}{0.4pt} \hspace{-0.1pt} \rule[5.1pt]{0.05pt}{0.4pt} \hspace{-0.1pt} \rule[5.15pt]{0.05pt}{0.4pt} \hspace{-0.1pt} \rule[5.2pt]{0.05pt}{0.4pt} \hspace{-0.1pt} \rule[5.25pt]{0.05pt}{0.4pt} \hspace{-0.1pt} \rule[5.3pt]{0.05pt}{0.4pt} \hspace{-0.1pt} \rule[5.35pt]{0.05pt}{0.4pt} \hspace{-0.1pt} \rule[5.4pt]{0.05pt}{0.4pt} \hspace{-0.1pt} \rule[5.45pt]{0.05pt}{0.4pt} \hspace{-0.1pt} \rule[5.5pt]{0.05pt}{0.4pt} \hspace{-0.1pt} \rule[5.55pt]{0.05pt}{0.4pt} \hspace{-0.1pt} \rule[5.6pt]{0.05pt}{0.4pt} \hspace{-0.1pt} \rule[5.65pt]{0.05pt}{0.4pt} \hspace{-0.1pt} \rule[5.7pt]{0.05pt}{0.4pt} \hspace{-0.1pt} \rule[5.75pt]{0.05pt}{0.4pt} \hspace{-0.1pt} \rule[5.8pt]{0.05pt}{0.4pt} \hspace{-0.1pt} \rule[5.85pt]{0.05pt}{0.4pt} \hspace{-0.1pt} \rule[5.9pt]{0.05pt}{0.4pt} \hspace{-0.1pt} \rule[5.95pt]{0.05pt}{0.4pt} \hspace{-0.1pt} \rule[6.0pt]{0.05pt}{0.4pt}}	

\numberwithin{equation}{section}

\usepackage{authblk}

\author[1]{Sara Daneri\thanks{sara.daneri@gssi.it}}
\author[2]{Emanuela Radici\thanks{emanuela.radici@epfl.ch}}
\author[3]{Eris Runa\thanks{eris.runa@gmail.com}}
\affil[1]{Gran Sasso Science Institute, L'Aquila, Italy}
\affil[2]{Ecole Polytechnique F\'ed\'erale de Lausanne, Switzerland}
\affil[3]{Deutsche Bank AG, London, UK}

  \title{Deterministic particle approximation of aggregation-diffusion  equations on unbounded domains }
  
\begin{document}
  	
  	\maketitle

  \begin{abstract}
  	We consider a one-dimensional  aggregation-diffusion equation, which is the gradient flow in the Wasserstein space of a functional with competing  attractive-repulsive interactions.

  	We prove that the fully deterministic particle approximations with piecewise constant densities  introduced in~\cite{difrancesco-rosini} starting from  general bounded initial densities converge strongly in $L^1$ to bounded  weak solutions of the PDE. 
  	
  	In particular, the result is achieved in unbounded domains and for arbitrary nonnegative bounded initial densities, thus extending the results in \cite{GT, MO, MS} (in which a no-vacuum condition is required) and giving an alternative approach to \cite{CCP} in the one-dimensional case, including also subquadratic and superquadratic diffusions.

  \end{abstract}

  \section{Introduction}
  In this paper we consider  the following aggregation-diffusion equation
  \begin{equation}
  \label{eq:mainPDE}
  \partial_t \rho = \partial_x\big(\rho\partial_x (K \ast \rho) +\partial_x\phi(\rho)\big), \qquad t\in[0,T],\, x\in \R.
  \end{equation}

  By  $K$ we denote an interaction kernel satisfying the following conditions
  \begin{align}
  	&K(z)=K(|z|);\label{eq:k1}\\
  	&K\in C^0(\R)\cap C^2(\R\setminus\{0\});\label{eq:k2}\\
  	&K,\,K'\,\in L^\infty(\R),\, K'' \in L^{\infty} (\R \setminus \{0\});\label{eq:k3}\\
  	&\|K'\|_{L^1(\R)}<\infty,\label{eq:k4}
  \end{align}
and by $\phi(\rho)=\rho W'(\rho)-W(\rho)$ a $C^1$ function satisfying 
\begin{align}
&{W \geq 0}\label{eq:phi1}\\
&\phi(0)=0\label{eq:phi2}\\
&\phi\text{ strictly monotone increasing }\label{eq:phi3}\\
&\phi'(\rho)\rho\leq c_0\phi(\rho)\label{eq:phi4}\\ &\phi(\rho)\leq\max\{\rho, c_0W(\rho)\}\label{eq:phi5}
\end{align}
for some constant $c_0>0$ and there exist $c_1,c_2>0$ such that
\begin{equation}\label{eq:phirhocond}
\phi(\rho)\leq c_1\rho\qquad\text{if $\rho\leq \hat{\rho}<1$},\qquad\text{and} \qquad\phi(\rho)\geq c_2\rho\qquad\text{if $\rho\geq \bar{\rho}>1$}.
\end{equation} 

The above assumptions cover in particular the case 
\begin{equation}
\phi(\rho)=\rho^m,\qquad m\geq1.
\end{equation}

 The equation~\eqref{eq:mainPDE} can be seen as the gradient flow in the Wasserstein space  of  a nonlocal interaction functional with attractive-repulsive terms in competition (see~\cite{CDFS} for potentials $K$ admitting as in our case a Lipschitz singularity at the origin)
  \begin{equation}
  \label{eq:EcalBeta}
  \Fcal(\rho) :=
  \frac12\int_{\R}\int_{\R} K(x-y) \rho(x)\rho(y)\dx\dy +\int_{\R}W(\rho)\dx. 
  \end{equation}
  
Notice our assumptions on $K$ do not require any particular attractive or repulsive behaviour. In particular, $K$ might be attractive at short range and repulsive at long range, with the nonlinear term $W$ acting as a repulsive term for large enough densities. An example for $K$ is the attractive-repulsive double Yukawa kernel considered in~\cite{DR2}, which is used in biology to model pattern formation in colloidal systems 
\begin{equation}\label{eq:Yuk}
K(z)=-\beta^2e^{-\beta |z|}+e^{-|z|},\quad \beta\gg1.
\end{equation}

In the case in which $\phi(\rho)=\rho^m$, $m>0$ and the kernel is of the form \eqref{eq:Yuk} (with no specific assumptions on the parameter $\beta$), global in time existence of bounded solutions to~\eqref{eq:mainPDE} was proved in~\cite{S}.

For equations without diffusion (i.e., $\phi=0$), the  local well-posedness with possibly finite blow-up of bounded solutions and long-time behaviour in case of global existence  was extensively studied in several papers, among which we recall~\cite{LT, L, BV, BCL, BL}. In \cite{CDFS} the authors show that in the case of $\lambda$-convex kernels, thus in particular under our assumptions, global weak measure solutions, allowing for concentration, exist and are unique.

In~\cite{kim-yao, craig, CKY, CT} in place of the diffusion, the repulsive effect is caused by an hard constraint on the density.
Global minimizers of the functional~\eqref{eq:EcalBeta} for purely attractive kernels and power law diffusions in the aggregation-dominated regime  have been studied in~\cite{bedrossian}.  
Stationary states and large-time behaviour of~\eqref{eq:mainPDE} for smooth short-range attractive kernels and quadratic diffusion have been extensively studied in~\cite{BDF, BDFF,DFY}. Uniqueness of solutions for smooth kernels was adressed in~\cite{BS}. 
In this paper the kernel $K$  might be instead only Lipschitz in the origin (in particular, $K''$ be only  a measure in the origin),  as in~\eqref{eq:Yuk}. In the case of purely attractive Yukawa kernels (also known as Morse kernels) an explicit formula for stationary solutions was derived in~\cite{DFY}.

In~\cite{CT} the authors show that the slow diffusion limit of the gradient flow of~\eqref{eq:EcalBeta} when $\phi(\rho)=\rho^m$ is given by the gradient flow of the functional~\eqref{eq:EcalBeta} with $W=0$ and an hard constraint on the density.  

For a recent review on the topic see~\cite{CCY}. 

In the case of the kernel with competing attractive-repulsive terms defined in~\eqref{eq:Yuk}, the $\Gamma$-limit of such a functional as $\beta\to+\infty$ (namely with local attractive term) has been characterized in~\cite{DR2}. 
In suitable regimes,  minimizers of the limit functionals have been proved to be given by periodic unions of stripes (i.e., intervals in one dimension,  see~\cite{DR2}) with techniques developed in~\cite{GR, DR1,  DR2, DKR, DR3}. For characterization of minimizers with power law attractive-repulsive potentials see~\cite{ChTo}.

\vskip 0.2 cm

In this paper we are interested in the well-posedness and convergence of a deterministic moving particle scheme approximation with piecewise constant densities for the aggregation-diffusion equation~\eqref{eq:mainPDE}. 

In absence of the diffusion term, in \cite{CCH} the authors give a deterministic approximation of the nonlocal PDE with empirical distributions weakly converging in the sense of measures, valid not only for our class of kernels but also for more singular kernels and in general dimension.

A deterministic approximation approach for linear and  nonlinear diffusion equations in which the diffusion operator is replaced by a nearest neighbour interaction term was introduced in~\cite{R1,R2}, mainly with numerical purpose.  In \cite{GT} the authors provide a deterministic particle scheme for \eqref{eq:mainPDE} in the case in which $K$ is convex and smooth and the initial density has no vacuum zones. In this case, they show pointwise convergence to a solution of the limit problem. In \cite{MO} and \cite{MS} the authors provide another particle approximation analogous  to the one used in our paper, with assumptions on the initial density similar to \cite{GT}. 

The deterministic particle approximation that we use in this paper was used in~\cite{difrancesco-rosini} for nonlinear scalar conservation laws of the type

\begin{equation}
   \label{eq:x1}
   \begin{split}
      \partial_t\rho+\partial_x(\rho v(\rho))=0,
   \end{split}
\end{equation}

with $v\in C^1([0,+\infty))$ and strictly decreasing.

If $v$ in~\eqref{eq:x1} is allowed to depend  both locally and non-locally on $\rho$ the derivation of such system of ODEs as microscopic Lagrangian formulation of the nonlinear scalar conservation law has been investigated in~\cite{piccoli-rossi} for Lipschitz continuous velocities. 

In~\cite{DFR} and~\cite{fagioli-radici-1} the authors use this approximation  in the setting of aggregation (resp. aggregation/diffusion) equations with nonlinear mobilities, namely for equations of the form 
\begin{equation}\label{eq:1.10}
\partial_t\rho=\partial_x(\mathcal M(\rho)\partial_x(K\ast\rho+a(\rho))).
\end{equation}
with $\mathcal M(\rho)=\rho v(\rho)$ with $v$ monotone decreasing and compactly supported. 

In presence of the diffusion term and assuming in addition that the initial density function is supported on an interval and bounded from below by a positive constant in its interior (i.e. the so-called no vacuum condition) in~\cite{fagioli-radici-1} the authors are able to prove $L^\infty$ bounds for the deterministic particle approximation and strong convergence to solutions of~\eqref{eq:1.10} on a bounded interval, with zero velocity boundary conditions.

 More recently, in~\cite{fagioli-radici-2}, the same deterministic scheme was used in the different context of opinion dynamics for \emph{space}-dependent mobilities {but still the argument is performed in bounded domains and with the no vacuum condition. }

In \cite{CCP} the authors introduce a general deterministic particle approximation (the so-called \emph{blob method}) for solutions of the multi-dimensional analogue of the PDE \eqref{eq:mainPDE} when $\phi(\rho)=\rho^m$ and $K$ is $C^1$, semiconvex and with at most quadratic growth. In particular, they provide sufficient regularity conditions on the approximation scheme that, if valid, give weak convergence of the approximations to the solutions of the PDE. Such conditions are proved to be valid if  $\phi(\rho)=\rho^m$ with $m=2$ and the kernel is $C^2$ on the whole space with uniformly bounded second derivatives, hence giving  weak convergence of the scheme.

\vskip 0.2 cm
Our approach provides a deterministic particle approximation for  arbitrary bounded and integrable initial densities in unbounded domains, extending thus the results of \cite{GT, MO, MS}. Our approximation is alternative in one space dimension to the one proposed in \cite{CCP} and moreover gives strong convergence of the scheme also for subquadratic and superquadratic  diffusions, not included in \cite{CCP}.
  
First of all we consider piecewise constant deterministic particle approximations to the PDE \eqref{eq:mainPDE} on the one-dimensional torus $\T_L=\R/L\Z$ for $L\gg1$
\begin{equation}\label{eq:mainPDETL}
 \partial_t \rho = \partial_x\big(\rho\partial_x (K \ast \rho) +\partial_x\phi(\rho)\big), \qquad x\in \T_L.
\end{equation}
 For those approximations  we prove global $L^\infty$ bounds (which are necessary so that such a scheme is well-defined)  independent on the number of particles and the size of the torus.  
 
 Basing on a discrete $W^{1,2}$ estimate on the nonlinear term we provide compactness estimates and get $L^1$-convergence to bounded weak solutions of~\eqref{eq:mainPDETL} on $\T_L$.  
 
 The advantage of our approach is that the estimates needed for strong compactness are independent of the size of the torus $L$ and can then be used to pass to the limit as $L\to\infty$ and get strong convergence to solutions to the PDE \eqref{eq:mainPDE}. Moreover, all the compactness estimates are independent of any lower bound on the density. Such a lower bound is only used to ensure that the limit densities obtained in the first limit as $N\to\infty$ on a fixed torus are indeed solutions of the PDE \eqref{eq:mainPDETL} and can be guaranteed by an approximation of the initial density with positive functions of uniformly bounded energy.

  Our main results are the following.

\begin{theorem}\label{thm:existenceapprox}[$L^\infty$ bound]\label{eq:prop:linfty}
	Let $\rho_{0,L}\in L^\infty(\T_L)$. Then, there exist constants $C(\Fcal_L(\rho_{0,L}),K, \|\rho_{0,L}\|_{L^\infty})$ and $ C(K,\Fcal_L(\rho_{0,L}))$ such that the following holds. For every $T>0$ there exists $\bar N=\bar N(L,T,\Fcal_L(\rho_{0,L}))$ s.t. the deterministic particle approximations $\{\rho^N_L\}_{N\geq\bar N}$ starting from $\rho_{0,L}$ are well-defined on $[0,T]$ and moreover
	\begin{equation}
	\sup_{N\geq\bar N}\sup_{t\in[0,T)}\|\rho^N_L(t)\|_{L^\infty(\T_L)}\leq C(\Fcal_L(\rho_{0,L}),K, \|\rho_{0,L}\|_{L^\infty},c_2)+ C(K,\Fcal_L(\rho_{0,L}),c_2)(1+T).
	\end{equation}
\end{theorem}

  As $N\to+\infty$ we have the following
  
  \begin{theorem}
  	\label{thm:convlfixed}
  	Let $\rho_{0,L}\in L^\infty(\T_L)$ be a bounded function satisfying
  	\begin{equation}
  	\label{eq:lowboundl}
  	\inf_{\T_L}\rho_{0,L}\geq\eps_L>0.
  	\end{equation}
  	Then, for every $T>0$ the deterministic particle approximations $\{\rho^N_L\}_{N\geq\bar N}$ constructed from $\rho_{0,L}$ in Theorem~\ref{thm:existenceapprox} converge (up to subsequences) in $L^1([0,T]\times\T_L)$ to a bounded probability density $\rho_L$ satisfying~\eqref{eq:mainPDE} on $\T_L$.
  \end{theorem}

Now identify with a slight abuse of notation functions defined on $\T_L$ with $L$-periodic functions on $\R$.
We have the following

\begin{theorem}
	\label{thm:convinl}
Consider initial data $\rho_0\in L^1(\R)\cap L^\infty(\R)$ with $\rho_0>0$ on $\R$ and $\int_{\R}|x|\rho_0(x)\dx<+\infty$ and define $\rho_{0,L}$ by cutting $\rho_0$ on $[-L/2,L/2)$ and extending it periodically. Then the functions $\rho_L\chi_{[-L/2,L/2)}$, being $\rho_L$ limits of the deterministic particle approximations starting from $\rho_{0,L}$ on $\T_L$ found in Theorem~\ref{thm:convlfixed}, converge (up to subsequences) in $L^1([0,T]\times\R)$ to a weak solution $\rho$ of the PDE~\eqref{eq:mainPDE} on $\R$.
\end{theorem}

Finally, approximating a general nonnegative initial datum $\rho_0$ with strictly positive initial data, we obtain the following

\begin{theorem}\label{thm:convfinal}
Let $\rho_{0}\in L^1(\R)\cap L^\infty(\R)$ such that $\int_{\R}|x|\rho_0(x)\dx<+\infty$ and let $\{\rho_{0}^\lambda\}_{\lambda\in\N}\subset L^1(\R)\cap L^\infty(\R)$ be a uniformly bounded  $L^1$ approximating sequence for $\rho_0$ with the property that $\rho_{0}^\lambda>0$ on $\R$ for all $\lambda\in\N$ and the first moments are uniformly bounded. 
Then, the sequence of solutions $\rho^\lambda$ of the PDE~\eqref{eq:mainPDE} found in Theorem~\ref{thm:convinl} from the initial  data $\rho_{0}^\lambda$ converges (up to subsequences) in $L^1([0,T]\times\R)$ to a solution of the PDE~\eqref{eq:mainPDE} starting from $\rho_0$.
\end{theorem}

   What prevents the blow-up of the deterministic particle approximations in our case are the following facts: 
   \begin{itemize}
   	\item the evolution of the deterministic particle approximations inside sublevels of the energy up to some $O(1/\sqrt{N})$ proved in Corollary~\ref{cor:energybound}
   	\item the estimate of a discrete $W^{1,2}$ norm of the function $\phi(\rho^N_L)$ through the time derivative of the functional along the deterministic particle flow. 
   	\item the special structure of the deterministic particle approximation, namely the relation between the values of the approximating function and the distance between particles.
   \end{itemize}
   We notice indeed that  a mere bound on the energy values along the evolution would not be sufficient to prevent the blow-up. Thus all three points above are fundamental ingredients in the proof.  
   
  The $L^\infty$ bounds are obtained differently from \cite{DFR, fagioli-radici-1}, due to the absence  of a nonlinear mobility which ``blocks'' density values larger than some given threshold.
   What gives us convergence of the scheme are not  uniform $BV$ bounds on the deterministic particle approximations as in~\cite{DFR,fagioli-radici-1, fagioli-radici-2}, which would be dependent on the size of the torus $\T_L$, but uniform $W^{1,2}$ bounds on the composed functions $\phi(\rho^N_L)$ together with the strict monotonicity assumption on the nonlinear  function $\phi$. Such bounds are obtained using an explicit computation of the derivative of the energy functional along the flow and the fact that sublevels of the energy are almost invariant under the particles evolution. With similar estimates one obtains also an H\"older-type continuity of the $1$-Wasserstein distance along the deterministic particle flow, similar to that obtained in \cite{MS},  which is needed in order to apply a  compactness argument.
   
   This allows us to have, in the limit as $N\to+\infty$, estimates which depend only on the energy and on the $L^\infty$ bounds of the initial data, thus allowing us to pass to the limit in an unbounded domain.
   \ \\

   The main novelty  present in this paper is that we are able to prove the strong convergence of the deterministic particle approximations scheme to the solutions of the aggregation-diffusion equation \eqref{eq:mainPDE} in an unbounded setting and for general  integrable and bounded initial densities.

\section{Preliminary facts}

  \subsection{Pseudo-inverse}
  
  Fix $x_0\in\T_L$. Given a nonnegative  density function $\rho:\T_L\to[0,+\infty)$ with $\int_{\T_L}\rho=c_L$, define its pseudo-inverse $X:[0,c_L]\to\T_L$ as follows
  \begin{equation}\label{eq:pseudoinv}
  X(z)=\sup\Big\{x:\int_{x_0}^x\rho(y)\dy<z\Big\}.
  \end{equation}
  If $\rho(t,x)\in L^\infty$ is a weak solution of~\eqref{eq:mainPDE} on $\T_L$ and $\rho>0$, then $X(t,z)=X_{\rho(t,\cdot)}(z)$ solves the PDE
  \begin{equation*}
  \partial_tX(t,z)=-\int_0^{c_L}K'(X(t,z)-X(t,\xi))\d\xi-\partial_z\phi(\rho(X(t,z))).
  \end{equation*}

  \subsection{Deterministic particle approximation}\label{sec:det}
  
  Our goal is to approximate~\eqref{eq:mainPDE} with a moving particle approximation on a series of increasing tori.
  
  Let us fix $L>0$, $N\in\N$ and $x_0\in\T_L$.

Given $\rho_{0,L}\in L^\infty(\T_L)$ with $\int_{\T_L}\rho_{0,L}=c_L\leq1$ and $\rho_{0,L}\geq0$, for all $k=1,\dots,N$ define (for the ordering identify $\T_L$ with $[-L/2,L/2]$ and set $x_0=-L/2$)
  \begin{equation*}
  x_k=\sup\Big\{x:\int_{x_{k-1}}^x\rho_{0,L}(y)\dy<\frac{c_L}{N}\Big\}.
  \end{equation*}
  Notice that $x_0<\cdots<x_{N-1}$  and $x_N=L/2=x_0+L$ (resp. $x_N=x_0$ on $\T_L$).

  For every $k=0,\dots,N$ define the following system of ODEs on $\T_L$
  \begin{align}\label{eq:ode}
  \dot x_k(t)=&-\frac{c_L}{N}\sum_{j\neq k}K'(x_k(t)-x_j(t))-\frac{ N}{c_L}[\phi(\rho_k(t))-\phi(\rho_{k-1}(t))],
  \end{align}
  where
  \begin{equation*}
  \rho_k(t)=\frac{c_L}{N(x_{k+1}(t)-x_k(t))}
  \end{equation*}
  and with initial conditions
  \begin{equation*}
  x_k(0)=x_k,
  \end{equation*}
  as long as $x_0(t)<\cdots<x_{N-1}(t)$. 
In the above the dependence on $N,L$ is omitted since clear from the context.

  Then define the deterministic particle approximations starting from $\rho_{0,L}$ as the piecewise constant functions 
  \begin{equation*}
  \rho^N_L(t,x)=\sum_{k=0}^{N-1}\rho_k(t)\chi_{[x_k(t),x_{k+1}(t))}(x).
  \end{equation*}
  We say that the deterministic particle approximation $\rho^N_L$ is well defined on $[0,T)$ provided the relation $x_0(t)<\dots<x_{N-1}(t)$ holds for all $t\in[0,T)$. 
  Notice that, despite of the above definition of starting point of the deterministic particle approximation, $\rho_{0,L}\neq\rho^N_L(0)$ but $\rho^N_L(0)\to\rho_{0,L}$ in $L^1$ as $N\to\infty$.
  
  Let $M_{\rho^N_L}:\T_L\to[0,c_L]$ be the cumulative distribution function of $\rho^N_L$, namely
  \[
  M_{\rho^N_L}(x)=\int_{x_0}^x \rho^N_L(y)\dy
  \]
   and $X:[0,c_L]\to\T_L$ its pseudoinverse function defined in~\eqref{eq:pseudoinv} for $\rho=\rho^N_L$, with $x_k=X(kc_L/N)$. 
  
Thus the piecewise constant approximations $\rho^N_L$  satisfy on $\T_L$ the PDE
  \begin{align}\label{eq:pdeapprox}
\partial_t\rho^N_L = & \partial_x\Big(\rho^N_L {K^\lin}'(\rho^N_L)\ast\rho^N_L+\frac{ N}{c_L}\rho^N_L\sum_k\chi_{[x_k(t),x_{k+1}(t))}(x)\bigl[\bigl(\phi(\rho_k)-\phi(\rho_{k-1})\bigr)\notag\\
&+\Bigl(\frac{ N}{c_L}M_{\rho^N_L(t)}(x)-k\Bigr)\bigl( (\phi(\rho_{k+1})-\phi(\rho_{k})) - (\phi(\rho_{k})-\phi(\rho_{k -1}))\bigr) \bigr]\Big),
\end{align} where
\begin{align*}
{K^\lin}'&(\rho^N_L)(x-y)=\sum_{k=0}^{N-1}\chi_{[kc_L/N,(k+1)c_L/N)}(M_{\rho^N_L}(x))\sum_{j=0}^{N-1}\chi_{[jc_L/N,(j+1)c_L/N)}(M_{\rho^N_L}(y))\cdot\notag\\
&\cdot\Bigl[(1-\chi_{\{0\}}(x_j-x_k)){K}'(x_k-x_{j})\notag\\
&+\Bigl(\frac{N}{c_L}M_{\rho^N_L}(x)-k\Bigr)[(1-\chi_{\{0\}}(x_j-x_{k+1}))K'(x_{k+1}-x_j)-(1-\chi_{\{0\}}(x_j-x_k))K'(x_{k}-x_j)]\Bigr].
\end{align*}
Due to the translation invariance of the torus, we can assume w.l.o.g. that $x_0(t)=x_N(t)$ is fixed during the evolution.

  \subsection{Bounds on the nonlocal interaction term}

  Since $\int_{\T_L}\rho^N_L=c_L$, for  any $z \in [0,c_L)$ one gets
  \begin{equation}\label{eq:kerbound}
   | K' \ast \rho^N_L (t, X(t,z))  | \leq c_L\|K' \|_{L^\infty}.
  \end{equation}

  We recall that $K'' \in L^{\infty}(\R \setminus \{0\})$, with a slight abuse of notation we denote from now on the uniform bound of $K''$ on its domain with $\| K'' \|_{L^\infty}$.

  If $z \in [kc_L/N, (k+1)c_L/N)$, then 
\begin{equation}\label{eq:kerLip}
   |K'\ast \rho^N_L(t,X(t,z)) - {K^\lin}'(\rho^N_L) \ast \rho^N_L (t,X(t,z))| \leq \| K'' \|_{L^\infty} (x_{k+1} - x_k) + c_L\frac{L\| K'' \|_{L^\infty} + 3 \| K' \|_{L^\infty}}{N}.
\end{equation}
  Indeed
  \begin{align*}
  |K' &\ast \rho^N_L(t,X(t,z)) - {K^\lin}'(\rho^N_L) \ast \rho^N_L (t,X(t,z))| \leq  \sum_{i \neq k} \rho_i(t) \int_{x_i}^{x_{i+1}} |K' (X(t,z) - y) - K' (x_k - x_i)| \dy \\
   &+ \sum_{i \neq k, k+1} \rho_i(t) \int_{x_i}^{x_{i+1}} \Bigl(\frac{N}{c_L}M_{\rho^N_L(t)}(X(t,z)) - k\Bigr) |K' (x_{k+1} - x_i) - K' (x_k - x_i)|  \dy \\
   &+ \frac{3c_L}{N}\| K'\|_{L^\infty} \\
\leq &  c_L \frac{\| K'' \|_{L^\infty}}{N} \sum_{i \neq k, k+1} [(X(z) - x_k) + (x_{i+1} - x_i)] + c_L\| K'' \|_{L^\infty}(x_{k+1} - x_k) + \frac{3c_L}{N}\| K'\|_{L^\infty},
  \end{align*}
and~\eqref{eq:kerLip} follows immediately since $\sum_i (x_{i+1} - x_i) = L$.

 \section{Estimates on the energy of the particle approximations}
 
  The aim of this section is to provide explicit computations and estimates on the energy of the discrete particle approximations. These will be crucial both in providing a uniform $L^\infty$ bound for the approximations and in finding the limit PDE~\eqref{eq:mainPDE}.
  
   In this section we will sometimes denote with $\rho^N_L(t)$ the function $\rho^N_L(t,\cdot):\T_L\to[0,+\infty)$. 
   
  We have the following estimate on the time derivative of the energy $\Fcal_L$ along the discrete particle approximations $\rho^N_L$, where
    \begin{equation}
  \label{eq:FcalL}
  \begin{split}
  \Fcal_L(\rho) :=\frac12\int_{\T_L}\int_{\T_L} K(x-y) \rho(x)\rho(y)\dx\dy +\int_{\T_L}W(\rho)\dx.
  \end{split}
  \end{equation}
  
 \begin{lemma}\label{lemma:decrease}
	Assume $\rho^N_L(t)$ is well defined on $[0,T)$. Then, for all $t\in[0,T)$ the functional $\Fcal_L(\rho^N_L(t))$ satisfies
	\begin{equation}\label{eq:gradient}
	\frac{d}{dt}\Fcal_L(\rho^N_L(t))\leq  \bar C(\| K'\|_{L^\infty}, \| K''\|_{L^\infty}) \frac{L}{\sqrt{N}}.
	\end{equation}
\end{lemma}

Estimate~\eqref{eq:gradient} is natural since the discrete particle approximations $\rho^N_L$ satisfy the PDE~\eqref{eq:pdeapprox}, which is an approximate version of the gradient flow in the Wasserstein space of the functional $\Fcal_L$.

A straightforward consequence of Lemma~\ref{lemma:decrease} is the following 
\begin{corollary}\label{cor:energybound}
 Let $T>0$, $\eps>0$, $L>0$. Then there exists $\bar N=\bar N(T,\eps,L)$ such that for all $N\geq \bar N$ if $\rho^N_L$ are the discrete particle approximations with initial datum $\rho_{0,L}$ and they are well defined on $[0,T)$, then 
 \begin{equation*}
	\Fcal_L(\rho^N_L(t))\leq \Fcal_L(\rho_{0,L})+2\eps \quad\forall\,t\in[0,T).
	\end{equation*}
\end{corollary}

\begin{proof}
	[Proof of Lemma~\ref{lemma:decrease}:]  
	
	For simplicity, in this proof we introduce the notation 
   \begin{equation}
      \label{eq:Fk}
      \begin{split}
	  F_k(t):= \phi (\rho_k(t)) -  \phi (\rho_{k-1}(t))
      \end{split}
   \end{equation}
	where 
	\[ \rho_k(t) = \rho^N_{L}(t,x_k(t)) = \frac{c_L}{N (x_{k+1}(t) - x_k(t))}.  \]
	
	From  the symmetry of the kernel, it is easy to see that 
	\begin{align}\label{eq:derform}
	 \frac{d}{dt}\Fcal_{L}(\rho^N_L(t)) =  \frac{d}{dt}\int_{\T_L}   W(\rho^N_L(t,x)) \dx +   \int_{\T_L}\int_{\T_L} K(x-y) \rho^N_L(t,y) \partial_t \rho^N_L(t,x) \dx\dy. 
	\end{align}

	Using~\eqref{eq:pdeapprox}, integration by parts and applying the pseudoinverse change of variables, one can rewrite the second term of the r.h.s. of~\eqref{eq:derform} as follows
	\begin{align*}
	  \int_{\T_L}\int_{\T_L} K(x-y) & \rho^N_L(t,y) \partial_t \rho^N_L(t,x) \dx\dy =  - \int_0^{c_L} \left( K' \ast \rho^N_L (t,X(t,z)) \right)^2 \dz \\
	 & + \int_0^{c_L} K'\ast \rho^N_L (t,X(t,z)) \big[K' \ast \rho^N_L (t,X(t,z)) -  {K^\lin}'(\rho^N_L) \ast \rho^N_L (t,X(t,z)) \big] \dz \\
	 & - \frac{N}{c_L} \sum_{k=0}^{N-1} \int_{kc_L/N}^{(k+1)c_L/N} K'\ast \rho^N_L (t,X(t,z))\Bigl[ F_k(t) + \Bigl(\frac{N}{c_L}z - k\Bigr)(F_{k+1}(t) - F_k(t)) \Bigr] \dz. 
	\end{align*}	
	
On the other hand, observing that the ODE~\eqref{eq:ode} can be rewritten as
\begin{equation*}
\dot x_k(t) =-\int_{\T_L} {K^\lin}'(\rho^N_L)(x_k-y)\rho^N_L(t,y)\dy-\frac{N}{c_L}F_k(t),
\end{equation*}
the first term of the r.h.s. of~\eqref{eq:derform} can be explicitly computed obtaining
\begin{align*}
 \frac{d}{dt}\int_{\T_L}&  W(\rho^N_L(t,x)) \dx =\frac{d}{dt}\sum_{k=0}^{N-1}(x_{k+1}-x_k) W(\rho_k)\\
 &= -\frac{N}{c_L} \sum_{k=0}^{N-1}  (F_k(t))^2 -  \sum_{k=0}^{N-1}  F_k(t)  {K^\lin}'(\rho^N_L)\ast \rho^N_L (t,x_k)   \\
&= -\frac{N^2}{c_L^2}  \sum_{k=0}^{N-1} \int_{kc_L/N}^{(k+1)c_L/N} \big( F_k(t)\big)^2 \dz -\frac{N}{c_L} \sum_{k=0}^{N-1} \int_{kc_L/N}^{(k+1)c_L/N}  F_k(t) K' \ast \rho^N_L (t,X(t,z)) \dz\\
 & - \frac{N}{c_L} \sum_{k=0}^{N-1} \int_{kc_L/N}^{(k+1)c_L/N}  F_k(t) [ K' \ast \rho^N_L (t,X(t,z))  -  {K^\lin}'(\rho^N_L) \ast \rho^N_L (t,x_k) ] \dz.
\end{align*}

Once here, we observe that the periodicity of the torus ensures that $\sum_kF_k^2=\sum_kF_{k+1}^2$ and as a consequence by standard computations
\[\sum_{k=0}^{N-1} \int_{kc_L/N}^{(k+1)c_L/N} (F_k(t))^2 \dz \geq \sum_{k=0}^{N-1} \int_{kc_L/N}^{(k+1)c_L/N}  \big(F_k(t) + \Bigl(\frac{N}{c_L}z-k\Bigr)(F_{k+1}(t) - F_k(t)) \big)^2 \dz.   \]
Then,
\begin{align}\label{eq:ender}
\notag
\frac{d}{dt}\Fcal_{L}(\rho^N_L(t)) \leq&  -\frac{1}{2} \sum_{k=0}^{N-1} \int_{kc_L/N}^{(k+1)c_L/N} \Bigl[ \frac{N}{c_L} F_k(t) + K'\ast \rho^N_L (t,X(t,z)) \Bigr]^2 \dz \\
\notag
&  -\frac{1}{2} \sum_{k=0}^{N-1} \int_{kc_L/N}^{(k+1)c_L/N} \Big[\frac{N}{c_L} \big(F_k(t) + \Bigl(\frac{N}{c_L}z -k\Bigr)(F_{k+1}(t) - F_k(t))\big)  + K' \ast \rho^N_L (t,X(t,z))\Big]^2 \dz \\
\notag
& -  \sum_{k=0}^{N-1} \int_{kc_L/N}^{(k+1)c_L/N} \frac{N}{c_L} F_k(t) [ K'_{} \ast \rho^N_L (t,X(t,z))  -  {K^\lin}'(\rho^N_L) \ast \rho^N_L (t,x_k) ]\dz \\
&+  \sum_{k=0}^{N-1} \int_{kc_L/N}^{(k+1)c_L/N} K' \ast \rho^N_L (t,X(t,z)) \big[K'\ast \rho^N_L (t,X(t,z)) -  {K^\lin}'(\rho^N_L) \ast \rho^N_L (t,X(t,z)) \big] \dz.
\end{align}
From now on, our aim is to show that the r.h.s. of~\eqref{eq:ender} is negative, up to a term which tends to zero as the number of particles goes to infinity (see~\eqref{eq:gradient}). Notice that the first two terms of the r.h.s. of~\eqref{eq:ender} are negative.

Let us now deal with the third term of the r.h.s. of~\eqref{eq:ender}.

From the Cauchy-Schwarz inequality and~\eqref{eq:kerLip} it is easy to see that 
\begin{align}\label{eq:us1}
\notag
 &\sum_{k=0}^{N-1} \int_{kc_L/N}^{(k+1)c_L/N}   \frac{N}{c_L}|F_k(t)| | K'\ast \rho^N_L (t,X(t,z))  -  {K^\lin}'(\rho^N_L) \ast \rho^N_L (t,x_k) | \dz \\
 \leq & \| K'' \|_{L^\infty} \left( \sum_{k=0}^{N-1} \frac{(x_{k+1} - x_k)\sqrt{c_L}}{\sqrt{N}} \right) \left( \sum_{k=0}^{N-1} \int_{kc_L/N}^{(k+1)c_L/N}  \Bigl|\frac{N}{c_L}F_k(t)\Bigr|^2 \dz \right)^{1/2} \notag \\
 & +\frac{ L\sqrt{c_L}C(\|K''\|_{L^\infty}, \|K'\|_{L^\infty})}{\sqrt{N}}\left( \sum_{k=0}^{N-1} \int_{kc_L/N}^{(k+1)c_L/N}  \Bigl|\frac{N}{c_L}F_k(t)\Bigr|^2 \dz \right)^{1/2}\notag\\
 \leq &  \frac{ L\sqrt{c_L} C(\|K''\|_{L^\infty}, \|K'\|_{L^\infty}) }{\sqrt{N}} \left( \sum_{k=0}^{N-1} \int_{kc_L/N}^{(k+1)c_L/N}  \Bigl|\frac{N}{c_L}F_k(t)\Bigr|^2 \dz \right)^{1/2}
\end{align}

Denoting for simplicity
\[  a_k = \frac{N}{c_L} F_k(t) , \qquad b_k = K' \ast \rho^N_L (t,X(t,z)) \chi_{[kc_L/N, (k+1)c_L/N)}(z),   \]
 we can distinguish two different cases: either 
\begin{equation}\label{eq:case2}
 |a_k|<2|b_k|
\end{equation}
or 
\begin{equation}   \label{eq:case1}
 |a_k|\geq2|b_k|.
\end{equation}

Let us now estimates the terms of the sum in~\eqref{eq:us1} where~\eqref{eq:case2} holds.

From~\eqref{eq:us1},~\eqref{eq:case2} and~\eqref{eq:kerbound} we deduce that 
\begin{align}
\label{eq:casoI}
 \sum_{|a_k|<2|b_k|} \int_{kc_L/N}^{(k+1)c_L/N} & \frac{N}{c_L}|F_k(t)| | K'\ast \rho^N_L (t,X(t,z))  -  {K^\lin}'(\rho^N_L) \ast \rho^N_L (t,x_k) | \dz\notag\\
 &\leq  \frac{ L\sqrt{c_L}C(\| K'' \|_{L^\infty},\|K'\|_{L^\infty})}{\sqrt{N}} \left( \sum_{|a_k|<2|b_k|} \int_{kc_L/N}^{(k+1)c_L/N}  |a_k|^2 \dz \right)^{1/2}\notag\\
 &\leq \frac{L\sqrt{c_L} C(\| K'' \|_{L^\infty},\|K'\|_{L^\infty})}{\sqrt{N}} \left( \sum_{|a_k|<2|b_k|} \int_{kc_L/N}^{(k+1)c_L/N}  (c_L)^2\|K'\|^2_{L^\infty} \dz \right)^{1/2}\notag\\
 &\leq \frac{ L(c_L)^{2}C(\| K'' \|_{L^\infty},\|K'\|_{L^\infty})}{\sqrt{N}}. 
\end{align}

Let us now deal with the terms of the sum in~\eqref{eq:us1} satisfying~\eqref{eq:case1}. By~\eqref{eq:us1},
\begin{align*}
 \sum_{|a_k|\geq2|b_k|} \int_{kc_L/N}^{(k+1)c_L/N} & \frac{N}{c_L}|F_k(t)| | K'\ast \rho^N_L (t,X(t,z))  -  {K^\lin}'(\rho^N_L) \ast \rho^N_L (t,x_k) | \dz\notag\\
&\leq  \frac{ L\sqrt{c_L} C(\|K''\|_{L^\infty}, \|K'\|_{L^\infty})}{\sqrt{N}} \left( \sum_{|a_k|\geq2|b_k|} \int_{kc_L/N}^{(k+1)c_L/N}  |a_k|^2 \dz \right)^{1/2}.
\end{align*} 
To estimate the above there are two cases. Either 
\[
\left(\sum_{|a_k|\geq2|b_k|} \int_{kc_L/N}^{(k+1)c_L/N}  |a_k|^2 \dz \right)^{1/2}\leq1
\]
and then 
\begin{align}\label{eq:casoII}
 &\sum_{|a_k|\geq2|b_k|} \int_{kc_L/N}^{(k+1)c_L/N}  \frac{N}{c_L}| F_k(t)| | K' \ast \rho^N_L (t,X(t,z))  -  {K^\lin}'(\rho^N_L) \ast \rho^N_L (t,x_k) | \dz\leq \frac{ L\sqrt{c_L}C(\|K''\|_{L^\infty}, \|K'\|_{L^\infty})}{\sqrt{N}}
\end{align}
or 
\[
\left(\sum_{|a_k|\geq2|b_k|} \int_{kc_L/N}^{(k+1)c_L/N}  |a_k|^2 \dz \right)^{1/2}>1
\]
and then using~\eqref{eq:case1}
\begin{align*}
 \sum_{|a_k|\geq2|b_k|} \int_{kc_L/N}^{(k+1)c_L/N} &  \frac{N}{c_L} |F_k(t)| | K' \ast \rho^N_L (t,X(t,z))  -  {K^\lin}'(\rho^N_L) \ast \rho^N_L (t,x_k) | \dz\leq\notag\\
 &\leq \frac{ L\sqrt{c_L}C(\|K''\|_{L^\infty}, \|K'\|_{L^\infty})}{\sqrt{N}} \sum_{|a_k|\geq2|b_k|} \int_{kc_L/N}^{(k+1)c_L/N}  |a_k|^2 \dz\notag\\
&\leq \frac{4  L\sqrt{c_L}C(\|K''\|_{L^\infty}, \|K'\|_{L^\infty})}{\sqrt{N}}\sum_{|a_k|\geq2|b_k|} \int_{kc_L/N}^{(k+1)c_L/N}  |a_k+b_k|^2 \dz
\end{align*}
Hence, in this case, if $N$ is large enough
\begin{align}\label{eq:casoIII}
- \sum_{|a_k|\geq2|b_k|} \int_{kc_L/N}^{(k+1)c_L/N} \Bigl\{\frac{1}{2}&\Bigl[ \frac{N}{c_L} F_k(t) + K' \ast \rho^N_L (t,X(z)) \Bigr]^2 \notag\\
&- \frac{N}{c_L}
F_k(t) [ K'\ast \rho^N_L (t,X(z))  -  {K^\lin}'(\rho^N_L) \ast \rho^N_L (t,x_k) ]\Bigr\}\dz\leq\notag\\
&\leq\sum_{|a_k|\geq2|b_k|}\int_{kc_L/N}^{(k+1)c_L/N}|a_k+b_k|^2 \Big[-\frac12+\frac{4  L\sqrt{c_L}C(\|K''\|_{L^\infty}, \|K'\|_{L^\infty})}{\sqrt{N}}\Big]<0.
\end{align}

Finally, by estimates~\eqref{eq:kerbound} and~\eqref{eq:kerLip}, and by the fact that
\[
 \sum_{k=0}^{N-1}(x_{k+1}-~x_k)=~L,
\] one has that
\begin{align}\label{eq:d/dtE}
\notag
&\sum_{k=0}^{N-1} \int_{kc_L/N}^{(k+1)c_L/N} |K' \ast \rho^N_L (t,X(t,z))| \bigl|K' \ast \rho^N_L (t,X(t,z)) -  {K^\lin}'(\rho^N_L) \ast \rho^N_L (t,X(t,z)) \bigr| \dz \\
\notag
&\leq \| K'\|_{L^\infty} \Bigl[ \frac{ \| K''\|_{L^\infty}L}{{N}} + \frac{L \| K''\|_{L^\infty} + 3\| K'\|_{L^\infty}}{N} \Bigr] \\
&\leq \bar C(\| K'\|_{L^\infty}, \| K''\|_{L^\infty}) \frac{L}{{N}}.
\end{align}

We conclude the proof starting from~\eqref{eq:ender} and gathering together~\eqref{eq:casoI},~\eqref{eq:casoII},~\eqref{eq:casoIII} and~\eqref{eq:d/dtE}. 
\end{proof}

\section{$L^\infty$ bound}

The aim of this section is to prove Theorem~\ref{thm:existenceapprox}.

We will need a series of preliminary lemmas. The main idea is to estimate the $L^\infty$ norm of the discrete particle approximations via an estimate on their discrete $W^{1,2}$ norm in the pseudo-inverse variables. In turn, estimates on such $W^{1,2}$ norm will be provided by the explicit formula for the derivative of the functional $\Fcal_L$ along the deterministic particle approximations computed in the previous section, together with the fact that the energy is essentially decreasing along the flow (see Corollary~\ref{cor:energybound}). By discrete $W^{1,2}$ norm of $\phi(\rho^N_L)$ in the pseudo-inverse variables we mean the quantity

\begin{equation}\label{eq:w12normdisc}
\frac{N}{c_L}\sum_{k=0}^{N-1}|\phi(\rho_{k+1})-\phi(\rho_k)|^2=\sum_{k=1}^{N-1}\int_{kc_L/N}^{(k+1)c_L/N}\frac{N^2}{c_L^2}|\phi(\rho_{k+1})-\phi(\rho_k(t))|^2\dz.
\end{equation}

One has the following Gronwall-type estimate on the time  derivative of the quantity~\eqref{eq:w12normdisc}.
\begin{lemma}\label{lemma:w12estder}
 Let $\rho_{0,L}\in L^1(\T_L)$ such that $\Fcal_L(\rho_{0,L})\leq C_0$. Let $T>0$ be such that  $\{\rho^N_L\}_N$, i.e. the discrete particle approximation starting from $\rho_{0,L}$, is well defined on $[0,T]$ and let $c_0$ the constant appearing in~\eqref{eq:phi4} and~\eqref{eq:phi5}. Then,  there exists $C=C(c_0, C_0, c_2,\bar{\rho}, K)$ and $\bar N=\bar N(L,T,C_0)$ such that for all $N\geq \bar N$ and  for all $t\in[0,T)$  one has that
 \begin{align}\label{eq:gronwder}
 \frac{d}{dt}&\sum_{k=1}^{N-1}\int_{kc_L/N}^{(k+1)c_L/N}\frac{N^2}{c_L^2}|\phi(\rho_{k+1}(t))-\phi(\rho_k(t))|^2\dz\leq\notag\\
 &\leq C(c_0, C_0,c_2,\bar{\rho}, K)+\frac{(2c_L\|K'\|_{L^\infty})^2}{c_2}\sum_{k=1}^{N-1}\int_{kc_L/N}^{(k+1)c_L/N}\frac{N^2}{c_L^2}|\phi(\rho_{k+1}(t))-\phi(\rho_k(t))|^2\dz.
 \end{align}

\end{lemma}

\begin{proof}
	One has that, using discrete integration by parts,
	\begin{align}
	 \frac{d}{dt}\sum_{k=1}^{N-1}\int_{kc_L/N}^{(k+1)c_L/N}&\frac{N^2}{c_L^2}|\phi(\rho_{k+1}(t))-\phi(\rho_k(t))|^2\dz=\notag\\
	 =&2\sum_{k=1}^{N-1}\int_{kc_L/N}^{(k+1)c_L/N}\frac{N^2}{c_L^2}[\phi(\rho_{k+1}(t))-\phi(\rho_k(t))][\phi'(\rho_{k+1}(t))\partial_t\rho_{k+1}(t)-\phi'(\rho_k(t))\partial_t\rho_k(t)]\dz\notag\\
	 =&-2\sum_{k=1}^{N-1}\int_{kc_L/N}^{(k+1)c_L/N}\frac{N^2}{c_L^2}(\phi(\rho_{k+1}(t))+\phi(\rho_{k-1}(t))-2\phi(\rho_k(t)))\phi'(\rho_k(t))\partial_t\rho_k(t)\dz.\label{eq:derest1}
	\end{align}
	By definition of $\rho_k$ one sees immediately that 
	\begin{equation}\label{eq:dtrhok}
	\partial_t\rho_k=-\rho_k^2\frac{N}{c_L}(\dot{x}_{k+1}-\dot{x}_k)
	\end{equation}
	Insert the above expression in~\eqref{eq:derest1} and   use the estimate 
	 \begin{equation*}
	\Bigl|\dot{x}_{k+1}-\dot{x}_k+\frac{N}{c_L}(\phi(\rho_{k+1})+\phi(\rho_{k-1})-2\phi(\rho_k))\Bigr|\leq\frac{c_L\|K'\|_{L^\infty}}{N}+\|K''\|_{L^\infty}(x_{k+1}-x_k),
	\end{equation*}
	which follows easily from the ODE~\eqref{eq:ode}.
	Then recalling that $\rho_k=\frac{c_L}{N(x_{k+1}-x_k)}$,
	 one obtains that
	\begin{align}
	 \frac{d}{dt}&\sum_{k=1}^{N-1}\int_{kc_L/N}^{(k+1)c_L/N}\frac{N^2}{c_L^2}|\phi(\rho_{k+1}(t))-\phi(\rho_k(t))|^2\dz\leq\notag\\
	  &\leq-2\sum_{k=1}^{N-1}\int_{kc_L/N}^{(k+1)c_L/N}\Bigl(\frac{N^2}{c_L^2}(\phi(\rho_{k+1}(t))+\phi(\rho_{k-1}(t))-2\phi(\rho_k(t)))\Bigr)^2\phi'(\rho_k(t))\rho_k^2(t)\dz\label{eq:derest2}\\
	  &+2c_L\|K'\|_{L^\infty}\sum_{k=1}^{N-1}\int_{kc_L/N}^{(k+1)c_L/N}\Bigl(\frac{N^2}{c_L^2}|\phi(\rho_{k+1}(t))+\phi(\rho_{k-1}(t))-2\phi(\rho_k(t))|\Bigr)\phi'(\rho_k(t))\rho_k^2(t)\dz\label{eq:derest3}\\
	  &+2\|K''\|_{L^\infty}\sum_{k=1}^{N-1}\int_{kc_L/N}^{(k+1)c_L/N}\Bigl(\frac{N^2}{c_L^2}|\phi(\rho_{k+1}(t))+\phi(\rho_{k-1}(t))-2\phi(\rho_k(t))|\Bigr)\phi'(\rho_k(t))\rho_k(t)\dz.\label{eq:derest4}
	  \end{align}
	  
	  Notice that by monotonicity of $\phi$ the terms in~\eqref{eq:derest2} are negative. These will be used to bound the terms in~\eqref{eq:derest3} and~\eqref{eq:derest4} up to a Gronwall-type inequality.

     Let us consider the term in~\eqref{eq:derest4}. Define
	  \begin{equation*}
	  A:=\Bigl\{z:\,\frac{N^2}{c_L^2}|\phi(\rho_{k+1}(t))+\phi(\rho_{k-1}(t))-2\phi(\rho_k(t))|\rho_k(t)\leq 2\|K''\|_{L^\infty}\Bigr\}.
	  \end{equation*} One has that
	  \begin{align}
	  &-\sum_{k=1}^{N-1}\int_{kc_L/N}^{(k+1)c_L/N}\Bigl(\frac{N^2}{c_L^2}(\phi(\rho_{k+1}(t))+\phi(\rho_{k-1}(t))-2\phi(\rho_k(t)))\Bigr)^2\phi'(\rho_k(t))\rho_k^2(t)\dz+\notag\\
	  &+2\|K''\|_{L^\infty}\sum_{k=1}^{N-1}\int_{kc_L/N}^{(k+1)c_L/N}\Bigl(\frac{N^2}{c_L^2}|\phi(\rho_{k+1}(t))+\phi(\rho_{k-1}(t))-2\phi(\rho_k(t))|\Bigr)\phi'(\rho_k(t))\rho_k(t)\dz\leq\notag\\
	  &\leq\sum_{k=1}^{N-1}\int_{[kc_L/N,(k+1)c_L/N]\cap A}(2\|K''\|_{L^\infty})^2\phi'(\rho_k(t))\dz\label{eq:derest5}\\
	  &\leq(2\|K''\|_{L^\infty})^2\int_{\T_L}\phi'(\rho^N_L(t,x))\rho^N_L(t,x)\dx\label{eq:derest6}\\
	  &\leq c_0(2\|K''\|_{L^\infty})^2\int_{\T_L}\phi(\rho^N_L(t,x))\dx\label{eq:derest7}\\
	  &\leq(2\|K''\|_{L^\infty})^2 \max \Bigl\{c_0,\,c_0^2\int_{\T_L}W(\rho^N_L(t,x))\dx\Bigr\}\label{eq:derest8}\\
	  &\leq c_0(2\|K''\|_{L^\infty})^2+ c_0^2(2\|K''\|_{L^\infty})^2\Bigl(\Fcal_L(\rho^N_L(0))+\frac{C(\|K'\|_{L^\infty}, \|K''\|_{L^\infty})L}{\sqrt{N}}T+\|K\|_{L^\infty}\Bigr)\label{eq:derest9}\\
	  &\leq \bar C(c_0, C_0, K)  
	  \end{align}
	  as soon as $N\geq\bar N(L,T,C_0)$. 
	  In the above inequalities we used the following: the monotonicity of $\phi$ from~\eqref{eq:derest5} to~\eqref{eq:derest6}; inequality~\eqref{eq:phi4} from~\eqref{eq:derest6} to~\eqref{eq:derest7}; inequality~\eqref{eq:phi5} from~\eqref{eq:derest7} to~\eqref{eq:derest8}; the estimate~\eqref{eq:gradient} from~\eqref{eq:derest8} to~\eqref{eq:derest9}; in the last estimate we used the bound on $\Fcal_L(\rho^N_L(0))$ by the constant $C_0$.

	  Let us now consider the term~\eqref{eq:derest3}. Define 
	
	 \begin{equation*}
	B:=\Bigl\{z:\,\frac{N^2}{c_L^2}|\phi(\rho_{k+1}(t))+\phi(\rho_{k-1}(t))-2\phi(\rho_k(t))|\leq 2c_L\|K'\|_{L^\infty}\Bigr\}.
	\end{equation*}
	
	One has that~
	\begin{align*}
 &-\sum_{k=1}^{N-1}\int_{kc_L/N}^{(k+1)c_L/N}\Bigl(\frac{N^2}{c_L^2}(\phi(\rho_{k+1}(t))+\phi(\rho_{k-1}(t))-2\phi(\rho_k(t)))\Bigr)^2\phi'(\rho_k(t))\rho_k^2(t)\dz+\notag\\
 &+2c_L\|K'\|_{L^\infty}\sum_{k=1}^{N-1}\int_{kc_L/N}^{(k+1)c_L/N}\Bigl(\frac{N^2}{c_L^2}|\phi(\rho_{k+1}(t))+\phi(\rho_{k-1}(t))-2\phi(\rho_k(t))|\Bigr)\phi'(\rho_k(t))\rho_k^2(t)\dz\notag\\
 &\leq\sum_{k=1}^{N-1}\int_{[kc_L/N,(k+1)c_L/N]\cap B}(2c_L\|K'\|_{L^\infty})^2\phi'(\rho_k(t))\rho_k^2(t)\dz.
 \end{align*}.
 By positivity of $\phi'$, \eqref{eq:phi4} and \eqref{eq:phirhocond} we get
 \begin{align}
 \sum_{k=1}^{N-1}\int_{[kc_L/N,(k+1)c_L/N]\cap B}&(2c_L\|K'\|_{L^\infty})^2\phi'(\rho_k(t))\rho_k^2(t)\dz\leq\notag\\
 &\leq\sum_{k=1}^{N-1}\int_{[kc_L/N,(k+1)c_L/N]}(2c_L\|K'\|_{L^\infty})^2\phi'(\rho_k(t))\rho_k^2(t)\dz\label{eq:c01}\\
 &\leq c_0\int_0^{c_L}(2c_L\|K'\|_{L^\infty})^2\phi(\rho^N_L(t, X(t,z)))\max\Bigl\{\bar{\rho}, \frac{\phi(\rho^N_L(t, X(t,z)))}{c_2}\Bigr\}\dz.\label{eq:c02}
\end{align}
Now we use the fact that due to the relation
\[
\sum_{k=0}^{N-1}(x_{k+1}-x_k)=L
\]
there exists always a $k=k(N,t)$ such that $\rho_k<\frac{c_L}{L}$ and we apply the following discrete Poincar\'e inequality
\begin{align*}
\int_0^{c_L}\phi(\rho^N_L(t, X(t,z)))&\max\Bigl\{\bar{\rho}, \frac{\phi(\rho^N_L(t, X(t,z)))}{c_2}\Bigr\}\dz\leq c_2\bar{\rho}^2+\frac{1}{c_2} \phi^2\Bigl(\frac{c_L}{L}\Bigr) \notag\\&+\frac{1}{c_2} \sum_{k=1}^{N-1}\int_{kc_L/N}^{(k+1)c_L/N}\frac{N^2}{c_L^2}|\phi(\rho_{k+1}(t))-\phi(\rho_k(t))|^2\dz.
\end{align*}

	\end{proof}

The following lemma relates the discrete $W^{1,2}$ norm of $\phi(\rho^N_L)$ in the pseudo-inverse variables with the values of the energy along the flow. 

\begin{lemma}\label{lemma:w12estint} For any $C_0>0$, $T>0$ there exist a constant $C(K, C_0)$ and $\bar N=\bar N(T, C_0,L)$ such that the following holds. Let $\rho_{0,L}\in L^1(\T_L)$ with $\Fcal_L(\rho_{0,L})\leq C_0$ and assume that  $\{\rho^N_L\}_{N\geq\bar N}$, i.e. the deterministic particle approximation starting from $\rho_{0,L}$, is well defined on $[0,T)$. Then, the following holds
		\begin{equation*}
	\sup_{N\geq \bar N}\int_0^T\sum_{k=0}^{N-1}N|\phi (\rho_k(t))-\phi(\rho_{k-1}(t))|^2\dt\leq C(K,C_0)(1+T).
	\end{equation*}

\end{lemma}

\begin{proof}
	
	Denoting for simplicity 
	\[a_k(t)= -\frac{N}{c_L}\big(\phi(\rho_{k}) - \phi(\rho_{k-1})\big)  \quad \,\,\mbox{ and } \quad\,\, b_k(t) = -\frac{c_L}{N} \sum_{i\neq k} K' (x_k - x_i),\]
	we observe that 
	\begin{align*}
	\frac{c_L}{N} \sum_{k=0}^{N-1} \int_0^T  |\dot{x}_k(t)|^2 \dt &= \int_0^T \frac{c_L}{N} \left[\sum_{k : \,|a_k(t)| > 2|b_k(t)|}  (a_k(t) + b_k(t))^2 +  \sum_{k: \,|a_k(t)| \leq 2|b_k(t)|} (a_k(t) + b_k(t))^2  \right]  \dt \notag \\
	&\geq \int_0^T \frac{ N}{4 c_L} \sum_{k : \,|a_k(t)| > 2|b_k(t)|}  |\phi(\rho_{k-1}) - \phi(\rho_k)|^2 \dt.
	\end{align*}
	
	From the calculation~\eqref{eq:ender} in  Lemma~\ref{lemma:decrease}, the time derivative of the discrete energy $\Fcal_L(\rho^N_L)(t)$ can be estimated from above by  
	\begin{equation*}
	\frac{d}{dt} \Fcal_L (\rho^N_L)(t) \leq - \frac{c_L}{N}\sum_{k=0}^{N-1} |\dot{x}_k|^2 +\int_0^{c_L} g_{N,L} (t, X(t,z)) \dz, 
	\end{equation*}
	where, thanks to the proof of Lemma~\ref{lemma:decrease}, we have that 
	\begin{equation}\label{eq:4.23}
	\int_0^{c_L} |g_{N,L} (t, X(t,z))| \dz \leq \bar{C}(\| K'\|_{L^\infty}, \| K''\|_{L^\infty}) \frac{L }{\sqrt{N}}.
	\end{equation}

	Thanks to Corollary~\ref{cor:energybound} with $\eps=C_0$ and the above estimate, when $N\geq\bar N$ and $\bar N$ is large enough depending on $L, T$ and $C_0$ we deduce the following bound 
	\begin{align}
	\frac{c_L}{N}\sum_{k=0}^{N-1} \int_s^t |\dot{x}_k|^2 \dt &\leq \Fcal_L(\rho^N_{L}(s)) - \Fcal_L(\rho^N_{L}(t)) + |s-t| \frac{{C}(\| K'\|_{L^\infty}, \| K''\|_{L^\infty})L }{\sqrt{N}}. \notag\\
	&\leq \Fcal_L((\rho_{0})_L) + 2C_0 - {\min \Fcal_L} + |s-t|\frac{{C}(\| K'\|_{L^\infty}, \| K''\|_{L^\infty})L }{\sqrt{N}}\notag\\
	&\leq C_0+2C_0-\min\mathcal F_L+|s-t|\frac{{C}(\| K'\|_{L^\infty}, \| K''\|_{L^\infty})L }{\sqrt{N}}\label{eq:new1}
	\end{align}

	Hence, by \eqref{eq:4.23} and \eqref{eq:new1} we have that
	\begin{equation}\label{eq:ab}
	N \int_0^T \sum_{k :\, |a_k(t)| > 2|b_k(t)|}  |\phi(\rho_{k}) - \phi(\rho_{k-1})|^2 \dt \leq  4 c_L\Big(3C_0 - {\min \Fcal_L} +T\frac{{C}(\| K'\|_{L^\infty}, \| K''\|_{L^\infty})L }{\sqrt{N}} \Big)
	\end{equation}
	On the other hand, $|a_k(t)| \leq 2 |b_k(t)|$ implies 
	\[   N^2| \phi(\rho_{k}) - \phi(\rho_{k-1}) |^2 < c_L^4 \| K' \|^2_{L^\infty}  \]
	thus also
	\begin{equation}\label{eq:ba}
	 N \int_0^T \sum_{k :\, |a_k(t)| \leq 2|b_k(t)|}  | \phi(\rho_{k}) - \phi(\rho_{k-1}) |^2 \dt  \leq T c_L^4 \| K' \|^2_{L^\infty}. 
	\end{equation}
	Now observe that 
	\begin{equation}\label{eq:minfnew}
	\inf_L\min\Fcal_L=\inf_L\,\,\inf_{\|\rho\|_{L^1(\R)}\leq c_L}\mathcal F_L(\rho)\geq-\|K\|_{L^\infty}>-\infty
	\end{equation}
	 where we used the fact that $W$ is nonnegative and  $K$ is uniformly bounded.
	Gathering together~\eqref{eq:ab},~\eqref{eq:ba} and~\eqref{eq:minfnew} one has that if $N\geq\bar N(T, C_0,L)$ is sufficiently large
	\[   
	\sup_N\int_0^T \sum_kN|\phi (\rho_k)-\phi(\rho_{k-1})|^2\dt\leq C(K,C_0)(1+T).
	\]
\end{proof}

As a corollary of Lemma~\ref{lemma:w12estder} and Lemma~\ref{lemma:w12estint} we have the following linear bound on the discrete $W^{1,2}$ norm of $\phi(\rho^N_L)$.

\begin{corollary}
	\label{cor:w12est}
	Let $C_0>0$, $T>0$. Then, there exist constants $C(K, C_0)$ and $C(c_0,C_0,K)$ and there exists $\bar N=\bar N(T, C_0,L)$ such that the following holds. Let $\rho_{0,L}\in L^1(\T_L)$ with $\Fcal_L(\rho_{0,L})\leq C_0$ and assume that  $\{\rho^N_L\}_{N\geq\bar N}$, i.e. the deterministic particle approximation starting from $\rho_{0,L}$, is well defined on $[0,T)$. Then, for all $0\leq t_0<t<T$
	\begin{align}\label{eqalin}
	N\sum_{k=0}^{N-1}|\phi(\rho_{k}(t))-\phi(\rho_{k-1}(t))|^2&\leq N\sum_{k=0}^{N-1}|\phi(\rho_{k}(t_0))-\phi(\rho_{k-1}(t_0))|^2 \notag\\
	&+ C(c_0,C_0,c_2,\bar{\rho}, K)(t-t_0)+C(K, C_0)\frac{(2c_L\|K'\|_{L^\infty})^2}{c_2}(1+(t-t_0)).
	\end{align}
\end{corollary}

\begin{remark}
	Notice that an exponential bound for the l.h.s. of~\eqref{eqalin} is given directly by Lemma~\ref{lemma:w12estder}. Indeed, the Gronwall-type inequality~\eqref{eq:gronwder} implies the following: for all $0\leq t_0<t<T$ it holds
	 \begin{align}
	\sum_{k=1}^{N-1}\int_{kc_L/N}^{(k+1)c_L/N}&\frac{N^2}{c_L^2}|\phi(\rho_{k+1}(t))-\phi(\rho_k(t))|^2\dz\leq\notag\\ \leq&\Bigl( \sum_{k=1}^{N-1}\int_{kc_L/N}^{(k+1)c_L/N}\frac{N^2}{c_L^2}|\phi(\rho_{k+1}(t_0))-\phi(\rho_k(t_0))|^2\dz+ \frac{C(c_0, C_0,c_2,\bar \rho, K)}{c(c_2,K)}\Bigl)e^{c(c_2,K)(t-t_0)},\label{eq:derestmain}
	\end{align}
	where we set $c(c_2,K)=(2\|K'\|_{L^\infty})^2/(c_2)$.
	
\end{remark}

In the following lemma we relate the $L^\infty$ norm of $\phi(\rho^N_L)$ to its discrete $W^{1,2}$ norm.

\begin{lemma}
	\label{lemma:w12linfty}
	Let $\rho^N_L(t)$ be the deterministic particle approximation at time $t$ defined starting from $\rho_{0,L}$ with $\int_{\T_L}\rho_{0,L}(x)\dx=c_L$. Then one has that 
	\begin{equation*}
	\|\phi(\rho^N_L(t))\|_{L^\infty(\T_L)}\leq\Bigl(\phi\Bigl(\frac{c_L}{L}\Bigr)+1\Bigr)+N\sum_{k=0}^{N-1}|\phi(\rho_{k+1}(t))-\phi(\rho_{k}(t))|^2.
	\end{equation*}
\end{lemma}

\begin{proof}
	The claim of the lemma follows immediately from the following facts: 
	\begin{enumerate}
	\item\[
		\|\phi(\rho^N_L(t))\|_{L^\infty(\T_L)}\leq |\phi(\rho_k)|+\sum_{j=0}^{N-1}|\phi(\rho_{k+j+1})-\phi(\rho_{k+j})|,\qquad\forall\,k\in\{0,\dots, N-1\}
	\]
	\item due to the relation
	\[
	\sum_{k=0}^{N-1}(x_{k+1}-x_k)=L
	\]
	there exists $k=k(N,t)$ such that $\rho_k(t)\leq\frac{c_L}{L}$
	\item it holds 
	\begin{equation}\label{eq:tvtv2ineq}
	\sum_{k=0}^{N-1}|a_{k+1}-a_k|\leq\max\Bigl\{1,N\sum_{k=0}^{N-1}|a_{k+1}-a_k|^2\Bigr\}.
	\end{equation}
		\end{enumerate}
\end{proof}

The following lemma guarantees the local existence of the discrete particle approximations on a time interval depending only on the $L^\infty$ norm of the initial datum.

\begin{lemma}\label{lemma:mininterv}
	Let $\rho_{0,L}\in L^\infty(\T_L)$. Then there exists $T_0=T_0(\|\rho_{0,L}\|_{L^\infty},\|K'\|_{L^\infty})>0$ such that the deterministic particle approximation $\{\rho^N_L\}_{N\in\N}$ defined starting from $\rho_{0,L}$ is well-defined on $[0,T_0)$.
\end{lemma}

\begin{proof}
	By definition of the deterministic particle approximations one has that
	\begin{equation*}
	x_{k+1}(0)-x_{k}(0)\geq\frac{c_L}{N\|\rho_{0,L}\|_{L^\infty}},\qquad\forall\,k=0,\dots,N-1.
	\end{equation*}
	Now we want to estimate from below the minimum time $T_0$ which is necessary in order to have $x_{k+1}(T_0)-x_{k}(T_0)=0$ for some $k\in\{0,\dots,N-1\}$.
	Observe that whenever $\rho_k(t)$ is maximum among all $\{\rho_{j}(t)\}_{j=1}^N$, then $\phi(\rho_{k+1})+\phi(\rho_{k-1})-2\phi(\rho_k)\leq0$ and therefore, assuming w.l.o.g. $\rho_{k}(t)>\|K''\|_{L^\infty}/(c_L\|K'\|_{L^\infty})$ one has that
	\begin{equation*}
	\dot{x}_{k+1}(t)-\dot{x}_k(t)\geq-c_L\frac{\|K'\|_{L^\infty}}{N}-\|K''\|_{L^\infty}\frac{c_L}{N\rho_k}\geq-\frac{2c_L\|K'\|_{L^\infty}}{N}.
	\end{equation*} 
	From this estimate, one has that $T_0\geq\frac{1}{2\|K'\|_{L^\infty}\|\rho_{0,L}\|_{L^\infty}}$.

\end{proof}

We conclude this section with the proof of Theorem~\ref{thm:existenceapprox}.

\begin{proof}[Proof of Theorem~\ref{thm:existenceapprox}:]
	By Lemma~\ref{lemma:mininterv} there exists $T_0=T_0(\|\rho_{0,L}\|_{L^\infty},\|K'\|_{L^\infty})$ such that the discrete particle approximations $\{\rho^N_L\}_{N\in\N}$ are well defined on $[0,T_0)$. Without loss of generality we can then assume that $T\geq T_0$. 
	By Lemma~\ref{lemma:w12estint} applied to the interval $[0,T_0)$ one has that there exists a constant $C(K,\Fcal_L(\rho_{0,L}))$ and $\bar N=\bar N(T_0,\Fcal_L(\rho_{0,L}), L)$ such that for all $N\geq\bar N$
	\begin{equation*}
	\sup_{N\geq\bar N}\int_{0}^{T_0}\sum_{k=0}^{N-1}N|\phi(\rho_k(t))-\phi(\rho_{k-1}(t))|^2\dt\leq C(K,\Fcal_L(\rho_{0,L}))(1+T_0).
	\end{equation*} 
	In particular, for all $N\geq \bar N$ there exists $t^N_{0}\in[0,T_0)$ such that
	\begin{equation*}
	\sum_{k=0}^{N-1}N|\phi(\rho_k(t^N_0))-\phi(\rho_{k-1}(t^N_0))|^2\leq \frac{C(K,\Fcal_L(\rho_{0,L}))(1+T_0)}{T_0}.
	\end{equation*}
	Assume now by contradiction that there exists $\bar T\in[T_0,T]$ such that the discrete particle approximations are well defined on $[0,\bar T)$ for sufficiently large $N$ but blow up at $\bar T$ for a sequence of arbitrarily large $N$.

	By Corollary~\ref{cor:w12est}, one has that for all $t\in[t^N_0,\bar T)$ and for all $N\geq\bar N$ with $\bar N$ eventually larger depending on $T$ it holds
	\begin{align*}
		N\sum_{k=0}^{N-1}|\phi(\rho_{k}(t))-\phi(\rho_{k-1}(t))|^2&\leq \frac{C(K,\Fcal_L(\rho_{0,L}))(1+T_0)}{T_0} \notag\\
	&+ C(c_0,C_0,c_2,\bar{\rho}, K)T+C(K, C_0)\frac{(2c_L\|K'\|_{L^\infty})^2}{c_2}(1+T).
	\end{align*}
	By Lemma~\ref{lemma:w12linfty}, for all $t\in[t^N_0,\bar T)$ one gets the $L^\infty$ bound
	\begin{align}
		\|\phi(\rho^N_L(t))\|_{L^\infty(\T_L)}&\leq\Bigl(\phi\Bigl(\frac{c_L}{L}\Bigr)+1\Bigr)+\frac{C(K,\Fcal_L(\rho_{0,L}))(1+T_0)}{T_0} \notag\\
		&+ C(c_0,C_0,c_2,\bar{\rho}, K)T+C(K,  C_0)\frac{(2c_L\|K'\|_{L^\infty})^2}{c_2}(1+T).\label{eq:phibound}
	\end{align}
	Thanks to the assumption~\eqref{eq:phirhocond} the bound~\eqref{eq:phibound} extends (up to a constant) to a similar bound for $\|\rho^N_L\|_{L^\infty(\T_L)}$ on $[0,\bar T)$. In particular, applying again Lemma~\ref{lemma:mininterv} to $\rho^N_L(\bar T-\eps)$ for some $\eps=\eps(\sup_{t<\bar T}\|\rho^N_L(t)\|_{L^\infty})$ sufficiently small, the discrete particle approximations  can be extended for any $N\geq \bar N$ up to the time $\bar T$ and even further, with the same bound. Thus a contradiction is reached and the statement of the theorem is proved. 
\end{proof}

\section{Convergence of the deterministic particle scheme}

The main goal of this section is to prove Theorem~\ref{thm:convlfixed},Theorem~\ref{thm:convinl} and~Theorem\ref{thm:convfinal}.

\subsection{$L^1$-Compactness}

In this paragraph we discuss the strong $L^1$-compactness in space and time of the following functions:
\begin{itemize}
	\item $\{\rho^N_L\}_{N\geq\bar N}$  (in order to prove Theorem~\ref{thm:convlfixed});
	\item $\{\rho_L\}_{L>0}$  (in order to prove Theorem~\ref{thm:convinl});
	\item $\{\rho^\lambda\}_{\lambda\in\N}$ (in order to prove Theorem~\ref{thm:convfinal}).
\end{itemize}

The proof of the various compactness results will not depend on the strict positivity of the initial densities, which will be  instead necessary to prove that the limit densities of the deterministic particle approximations $\rho^N_L$ on $N\to\infty$ are solutions  of the PDE.

In order to show compactness of the approximate solutions we will use the following generalized Aubin-Lions Lemma given in Theorem 2 of~\cite{RossiSav}. Before recalling it, we need to introduce the following definitions. 

Let $X$ be a separable Banach space.  We recall that a functional $\mathcal G: X\to[0,+\infty]$ is a \emph{normal integrand} if it is l.s.c. with respect to the Borel $\sigma$-algebra $\mathcal B(X)$.
$\mathcal G$ is also \emph{coercive} if the sublevels $\{v \in X : \mathcal G (v) \leq c\}$ are compact for any $c\geq0$. 

A pseudo-distance $g : X \times X \to [0, +\infty]$ is compatible with $\mathcal G$  if for every $v,w$ such that $g(v,w)=0$ and $\mathcal G(v)<+\infty$, $\mathcal G(w)<+\infty$ then $v=w$.

We are ready to recall Theorem 2 of~\cite{RossiSav} in a simplified form which is sufficient for our purposes.

\begin{theorem}\cite{RossiSav}\label{thm:aubinlions}
	Let $X$ be a separable Banach space. Let $\mathcal U$ be a set  of measurable functions $v:(0,T)\to X$, let $\mathcal G : X \to [0, +\infty]$ be a normal coercive integrand  and 	$g : X \times X \to [0, +\infty]$ be a l.s.c. pseudo-distance compatible with $\mathcal G$. Assume moreover that
	\begin{equation}\label{eq:tight}
	\sup_{v\in\mathcal U}\int_0^T\mathcal G(v(t))\dt<+\infty
	\end{equation}
	and
	\begin{equation}\label{eq:gholder}
	\lim_{h\to0}\sup_{v\in\mathcal U}\int_0^{T-h}g(v(t+h),v(t))\dt=0.
	\end{equation}
	Then $\mathcal U$ contains a sequence $v_n$ which converges in measure (w.r.t. $t\in(0,T)$ and with values in $X$) to a limit $v:~(0,T)\to~X$.
	
\end{theorem}

In order to state the various compactness results of this section, let us then fix $C_0,C_1>0$ and let us consider any $\rho_0:\R\to[0,+\infty)$ such that $\|\rho_0\|_{L^\infty}\leq C_1/4$, $\Fcal(\rho_0)\leq C_0/4$ and $\int_{\R}\rho_0=1$.

For any measurable function $g:\R\to[0,+\infty)$ we define 
 \begin{equation}\label{eq:gL}
g_L:[-L/2,L/2]\to[0,+\infty),\qquad g_L(x)=g(x)\chi_{[-L/2,L/2]}(x).
\end{equation}
 With a slight abuse of notation in this subsection we will still denote with $g_L$ the corresponding periodic extension function on $\T_L$. 
 
 Let $\rho^N_L$ be the deterministic particle approximation on $\T_L$ starting from the density $(\rho_{0})_L$, with $(\rho_0)_L$ defined as~\eqref{eq:gL} from $\rho_0$. W.l.o.g., we can assume that $N$ and $L$ are sufficiently large so that
 \begin{align}\label{eq:lambdacond}
 \sup_{N,L}\bigl\|\rho^N_L(0)\bigr\|_{L^\infty}\leq C_1,\qquad \sup_{N,L}\Fcal_L(\rho^N_L(0))\leq C_0.
 \end{align}
 
 Theorem~\ref{thm:existenceapprox} guarantees that for every $T>0$ the functions $\rho^N_L$ are well defined on $[0,T]$ as soon as $N\geq\bar N$ with $\bar N=\bar N(T, C_0,L)$ and that they enjoy the following bound
 
 \begin{equation}\label{eq:gamma1gamma2}
 \sup_{N\geq\bar N}\sup_{t\in[0,T]}\|\rho^N_L(t)\|_{L^\infty(\T_L)}\leq\gamma_1+\gamma_2T,
 \end{equation}
where $\gamma_1=\gamma_1(K,C_0,C_1,c_0,c_2)$ and $\gamma_2=\gamma_2(K,C_0,c_2)$.

Our first aim is to show the following
\begin{theorem}\label{thm:compactness}
	Let $\rho_0\in L^1(\R)\cap L^\infty(\R)$. Then, for all $T>0$ the deterministic particle approximation $\{\rho_L^N\}_{N\in\N}$ defined on $[0,T]\times\T_L$ starting from $(\rho_0)_L$ converges up to subsequences as $N\to+\infty$ to a function $\rho_L$ in $L^1([0,T] \times \T_L)$. 
\end{theorem}

In order to do so,  let us   define, for a function $v\chi_{[-L/2,L/2]} \in L^1(\R)$, the quantity
\begin{align*}
TV(v) &: = \sup_{I\in\N\,\,}\sup_{-L/2=x_0<x_1<\dots<x_I=L/2}\sum_{i=0}^{I-1} |v(x_{i+1})-v(x_i)|.
\end{align*}

Let us observe, that $TV(v)$ corresponds to the standard total variation of the function $v$.

Let us then define  $\mathcal G$ as follows
\begin{equation}\label{eq:Gdef}
\mathcal G(v)= TV(v)+\|v\|_{L^1(\R)}+\mathbb 1_{\{\|v\|_{L^\infty}\leq \phi(\gamma_1+\gamma_2T)\}}+\mathbb 1_{\Big\{\mathtt d_{W_1}\Big(\frac{\phi^{-1}(v)}{\|\phi^{-1}(v)\|_{L^1}}, \rho_0\Big)\leq \Lambda\Big\}},
\end{equation}
where 
\begin{equation*}
\mathbb 1_A(x)=\left\{\begin{aligned}
&0 && &\text{if $x\in A$}\\
&+\infty && &\text{if $x\in A^c$},
\end{aligned}\right.
\end{equation*}
 $\mathtt d_{W_1}$ is the standard 1-Wasserstein distance between probability measures and $\Lambda$ is a suitable positive constant to be defined later (see Remark~\ref{rem:lambda}).

Moreover, define
\[
g(v,w)=\mathtt d_{W_1} \left(\frac{\phi^{-1}(v)}{\|\phi^{-1}(v)\|_{L^1}},\frac{\phi^{-1}(w)}{\|\phi^{-1}(w)\|_{L^1}}\right).
\] 
 
It is fairly easy to see that $\mathcal G$ is a normal coercive integrand. Indeed, the compactness of the sublevels in $L^1_{\loc}$ comes from the first three terms of $\mathcal G$ and gets upgraded to compactness in $L^1$ due the tightness given by the last term in the definition of $\mathcal G$ and to the condition~\eqref{eq:phirhocond}. Moreover, $g$ is a l.s.c. pseudo-distance compatible with $\mathcal G$.

Let us now show that Theorem~\ref{thm:aubinlions} can be applied to the set of functions $\mathcal U=\{\phi(\rho^N_L)\}_{N\in\N}$ and the functionals $\mathcal G$ and $g$ defined above.

Concerning property~\eqref{eq:gholder}, we have the following

\begin{lemma}\label{lemma:g}
	Let $\rho_L^N : [0,T] \times \T_L\to(0,\infty)$, $N\geq\bar N(T,L,C_0)$ be the deterministic particle approximations starting from $(\rho_{0})_L$ and $t,s\in[0,T]$. Then there exists a constant $C(C_0,K)$ independent of $N,L$ such that 
	\begin{equation}\label{eq:timecont}
	\mathtt d_{W_1} \left(\frac{\rho_L^N(s)}{\|\rho_L^N(s)\|_{L^1}}, \frac{\rho_L^N(t)}{\|\rho_L^N(t)\|_{L^1}} \right)\leq C(C_0,K) |t-s|^{1/2}.
	\end{equation} 
\end{lemma}

\begin{proof} 
		Denote by $X(\tau)=X(\tau,\cdot)$ the pseudo-inverse of $\rho^N_L(\tau)$ and let $s<t$.  Let $c_L=\int_{\T_L}(\rho_{0})_L=\int_{\T_L}\rho^N_L(\tau)$.
	In order to estimate the $1$-Wasserstein distance of the deterministic particle approximations at different times we use the well-known identity
	\[
		\mathtt d_{W_1} \left({\rho_L^N(s)}, {\rho_L^N(t)} \right)=\|X(s)-X(t)\|_{L^1([0,c_L])}
	\]
 One has that
		\begin{align*}
	\|X(s)&-X(t)\|_{L^1([0,c_L])}^2	\leq c_L \|X(s)-X(t)\|^2_{L^2([0,c_L])}\notag\\
	&\leq c_L\sum_{k=0}^{N-1}\int_{k c_L/N}^{(k+1) c_L/N}|x_k(t)-x_k(s)+(2Nz-k)[x_{k+1}(t)-x_{k}(t)-x_{k+1}(s)+x_k(s)]|^2\dz\notag\\
	&\leq\sum_{k=0}^{N-1}\frac{c_L}{N}|x_k(t)-x_k(s)|^2+\frac{c_L^2}{N^2}(|x_{k+1}(t)-x_k(t)|^2+|x_{k+1}(s)-x_k(s)|^2).
	\end{align*}
	Moreover,
	\begin{align*}
	|x_{k}(t)-x_k(s)|^2=\Big|\int_s^t\dot x_k(\tau)\d\tau\Big|^2\leq|t-s|\int_s^t|\dot x_k(\tau)|^2\d\tau,
	\end{align*}
	hence 
	\begin{align}\label{eq:ftauest}
		\|X(s)-X(t)\|_{L^1}^2&\leq 2|t-s|\int_s^t\frac{c_L}{N}\sum_{k=0}^{N-1}|\dot x_k(\tau)|^2\d\tau
	\end{align}
	provided $N\gg L$. 
	
From the calculation~\eqref{eq:ender} in  Lemma~\ref{lemma:decrease}, the time derivative of the discrete energy $\Fcal_L(\rho^N_L)(t)$ can be estimated from above by  
\begin{equation*}
\frac{d}{dt} \Fcal_L (\rho^N_L)(t) \leq - \frac{c_L}{N}\sum_{k=0}^{N-1} |\dot{x}_k|^2 +\int_0^{c_L} g_{N,L} (t, X(t,z)) \dz, 
\end{equation*}
where, thanks to the proof of Lemma~\ref{lemma:decrease}, we have that 
\begin{equation*}
\int_0^{c_L} |g_{N,L} (t, X(t,z))| \dz \leq \bar{C}(\| K'\|_{L^\infty}, \| K''\|_{L^\infty}) \frac{L }{\sqrt{N}}.
\end{equation*}

Thanks to Corollary~\ref{cor:energybound} with $\eps=C_0$ and the above estimate, when $N$ is large enough depending on $L$ we deduce the following bound 
\begin{align*}
\frac{c_L}{N}\sum_{k=0}^{N-1} \int_s^t |\dot{x}_k|^2 \dt &\leq \Fcal_L(\rho^N_{L}(s)) - \Fcal_L(\rho^N_{L}(t)) + |s-t|\frac{{C}(\| K'\|_{L^\infty}, \| K''\|_{L^\infty})L }{\sqrt{N}} \notag\\
&\leq \Fcal_L((\rho_{0})_L) + 2C_0 - {\min \Fcal_L} + |s-t|\frac{{C}(\| K'\|_{L^\infty}, \| K''\|_{L^\infty})L }{\sqrt{N}}\notag\\
&\leq C_0+2C_0-\min\mathcal F_L+|s-t|\frac{{C}(\| K'\|_{L^\infty}, \| K''\|_{L^\infty})L }{\sqrt{N}}
\end{align*}	

which, applied to~\eqref{eq:ftauest}, implies 
\[
\|X(s)-X(t)\|_{L^1}^2\leq2c_L|t-s|\Big(3C_0 -{\min \Fcal_L} +T \frac{{C}(\| K'\|_{L^\infty}, \| K''\|_{L^\infty})L }{\sqrt{N}}\Big).
\]
and~\eqref{eq:timecont} follows from the fact that 
\begin{equation*}
\inf_L\min\Fcal_L=\inf_L\,\,\inf_{\|\rho\|_{L^1(\R)}\leq c_L}\mathcal F_L(\rho)\geq-\|K\|_{L^\infty}>-\infty
\end{equation*}
{where the last lower bound holds true since $W$ is non negative and  $K$ is uniformly bounded.}
\end{proof}

\begin{remark}
	\label{rem:lambda} Let us assume that $\rho_0$ has finite first moments. 
	Since the functions $\rho^N_L(0)$ converge in $L^1(\T_L)$ to $(\rho_{0})_L$, which in turn as $L\to+\infty$ converge to $\rho_0$ in $L^1(\R)$,  by Lemma~\ref{lemma:g} 
	\[
	\mathtt d_{W_1}\Big(\frac{\rho^N_L(t)}{\|\rho^N_L(t)\|_{L^1}}, \rho_0\Big)\leq C(C_0,K)T^{1/2}+\bar C,\quad\forall\,t\in[0,T]
	\] 
with $\bar C$ independent of $N,L$ as soon as $N\geq \bar N(L)$ is sufficiently large. Therefore the constant $\Lambda$ in definition~\eqref{eq:Gdef} of $\mathcal G$  can be chosen such that 
	\begin{equation*}
	\Lambda\geq 2(C(C_0,K)T^{1/2}+\bar C),
	\end{equation*}
	so that all the deterministic particle approximations (and their limits as $N\to\infty$ on $[0,T]$, by lower semicontinuity of $\mathtt d_{W_1}$) satisfy the condition 
	\[
		\mathtt d_{W_1}\Big(\frac{\rho^N_L(t)}{\|\rho^N_L(t)\|_{L^1}}, \rho_0\Big)\leq \Lambda.
	\]
	The above condition  is necessary for $\mathcal G$ to be finite along the deterministic particle evolution. Moreover, it gives tightness of the deterministic particle approximations and of their limits. Thanks to~\eqref{eq:phirhocond}, this converts into tightness for the functions $\phi(\rho^N_L)$ and their limits.
	
\end{remark}

\begin{remark}\label{rem:linfty} Observe that, by the assumptions~\eqref{eq:lambdacond} and~\eqref{eq:gamma1gamma2} (see Theorem~\ref{thm:existenceapprox}), the functions $\phi(\rho^N_L)$ satisfy the upper bound 
	\begin{equation}\label{eq:linftyboundphi}
	\sup_{N\geq\bar N}\sup_{t\in[0,T]}\|\phi(\rho^N_L(t))\|_{L^\infty}\leq \phi(\gamma_1+\gamma_2T)
	\end{equation}
	and therefore, also in view of Remark~\ref{rem:lambda}, they  lie in the domain of the functional $\mathcal G$.
\end{remark}

Let us now prove~\eqref{eq:tight} for the functional $\mathcal G$ defined in~\eqref{eq:Gdef} on the functions $\phi(\rho^N_L)$. First of all, one has that by~\eqref{eq:phi5}, the boundedness of $K$ and Corollary~\ref{cor:energybound}

\begin{align}
\sup_{N\geq\bar N}\int_0^T\int_{\T_L}\phi(\rho^N_L(t,x))\dx\dt&\leq\sup_{N\geq\bar N}\int_0^T\int_{\T_L}c_0W(\rho^N_L(t,x))\dx\dt+c_LT\notag\\
&\leq c_0C_0T+c_0\|K\|_{L^\infty}T+c_LT\notag\\
&\leq C(c_0, C_0, K)T<+\infty.\label{eq:phil1bound}
\end{align}

Moreover, one has the following

\begin{equation}\label{eq:bvest}
\sup_{N\geq\bar N}\int_0^T TV(\phi(\rho^N_L)(t))\dt \leq 2C(K, C_0)(1+T).
\end{equation}

Indeed, by Lemma~\ref{lemma:w12estint} one has that

\begin{equation*}
\sup_{N\geq\bar N}\int_0^T \sum_{k=0}^{N-1}N|\phi(\rho_{k+1}(t))-\phi(\rho_k(t))|^2\dt \leq C(K, C_0)(1+T).
\end{equation*}

Thus, applying~\eqref{eq:tvtv2ineq} we deduce~\eqref{eq:bvest}.

From Lemma~\ref{lemma:g}, Remarks~\ref{rem:lambda} and~\ref{rem:linfty} and the bounds~\eqref{eq:phil1bound} and~\eqref{eq:bvest} we deduce that for every $L$ the set
\[
\mathcal U=\{\phi(\rho^N_L)\}_{N\geq\bar N(T,L, C_0)}
\]
satisfies the assumptions of Theorem~\eqref{thm:aubinlions} on $X=L^1(\T_L)$. Hence Theorem~\ref{thm:aubinlions} can be applied, implying the convergence in measure (w.r.t. $t$ with values in $L^1(\T_L)$ and up to subsequences) of the functions $\phi(\rho^N_L):(0,T)\times \T_L\to\R$ to a function $\bar{\phi}_L$. 

By the $L^\infty$ bound~\eqref{eq:linftyboundphi}, the above convergence can be upgraded to convergence in $L^1([0,T]\times\T_L)$. 

By Theorem~\ref{thm:existenceapprox}, also the sequence  $\{\rho^N_L\}_{N\geq \bar N}$ is uniformly bounded on $[0,T]$. In particular, the functions  $\rho^N_L$ converge weakly* in $L^\infty([0,T]\times\T_L)$ (up to subsequences) to some bounded function $\rho_L$. 

By the strict monotonicity and continuity of $\phi$ and $\phi^{-1}$ it is not difficult to deduce (for example looking at the Young measures generated by subsequences of $\rho^N_L$ and $\phi(\rho^N_L)$) that $\bar{\phi}_L=\phi(\rho_L)$ and that $\rho^N_L$ converges strongly to  $\rho_L$ in $L^1([0,T]\times\T_L)$.

This concludes the proof of Theorem~\ref{thm:compactness}.

\vskip 0.3 cm

In order to conclude this section we show that also the following compactness result holds.

\begin{theorem}
	\label{thm:compactness2} Let $\rho_0\in L^1(\R)\cap L^\infty(\R)$ with finite first moments and let  $\rho_L:[0,T]\times[-L/2,L/2]\to[0,+\infty)$ be the functions obtained in Theorem~\ref{thm:compactness}. Then, as $L\to+\infty$, $\rho_L$ converge (up to subsequences) strongly in $L^1([0,T]\times\R)$ to a bounded function $\rho:[0,T]\times\R\to[0,+\infty)$. 
	
	Moreover, let 
	$\{\rho_0^\lambda\}_{\lambda\in\N}$ with $\rho_{0}^\lambda:\R\to(0,+\infty)$ such that $\|\rho_0^\lambda\|_{L^\infty}\leq 2\|\rho_0\|_{L^\infty}$, $\Fcal(\rho_0^\lambda)\leq 2\Fcal(\rho_0)$, $\int_{\R}\rho_0^\lambda=1$, the first moments are uniformly bounded  and 
	\begin{equation*}
	\rho^\lambda_0\longrightarrow\rho_0\quad\text{in $L^1$ as $\lambda\to\infty$}.
	 \end{equation*} Then, as $\lambda\to+\infty$, the densities  $\rho^\lambda$ found as the $\rho$ above but starting from $\rho_0^\lambda$ instead of $\rho_0$  converge (up to subsequences) strongly in $L^1([0,T]\times\R)$ to a bounded  density $\bar \rho$.
\end{theorem}

In order to prove Theorem~\ref{thm:compactness2} it is sufficient to observe that the estimates of Lemma~\ref{lemma:g}, Remarks~\ref{rem:lambda} and~\ref{rem:linfty} and the upper bounds~\eqref{eq:phil1bound} and~\eqref{eq:bvest} do not depend on $N\geq\bar N,L$ but only on  $C_0$, $C_1$ as in~\eqref{eq:lambdacond}.  Moreover, by lower semicontinuity of the total variation and of the $1$-Wasserstein distance $\mathtt d_{W_1}$ such estimates  uniformly hold also for the sequences $\rho_L$,  $\phi(\rho_L)$. Hence it is possible to apply Theorem~\ref{thm:aubinlions} to the sequence $\phi(\rho_L)$ on $[0,T] \times \R$ and repeat the previous reasoning obtaining a strong $L^1$ limit $\rho:[0,T]\times\R\to[0,+\infty)$.

If $\rho^\lambda_0$ is a sequence of initial data as in the statement of the Theorem, one can assume w.l.o.g. that 
\begin{align*}
\sup_{N,L,\lambda}\bigl\|\bigl(\rho^\lambda_0\bigr)^N_L\bigr\|_{L^\infty}\leq C_1,\qquad \sup_{N,L,\lambda}\Fcal_L\bigl(\bigl(\rho^\lambda_0\bigr)^N_L\bigr)\leq C_0
\end{align*}
where $C_0$ and $C_1$ are as in~\eqref{eq:lambdacond}.

Therefore, since $C_0$ and $C_1$ are independent of $\lambda$, the estimates for the applicability of Theorem~\ref{thm:aubinlions} to the sequence $\phi((\rho^\lambda)^N_L)$ are independent of $\lambda$, thus giving as above (up to subsequences) strong $L^1$ limits in $[0,T]\times \R$  $\rho^\lambda_L$ (as $N\to+\infty$), $\rho^\lambda$ (as $L\to+\infty$) and $\bar\rho$ (as $\lambda\to+\infty$). 

.

\begin{remark}
Notice that, while in the first limit (namely $L$ fixed and $N\to\infty$) we could have restricted to functions in $X=L^1(\T_L)$ and avoided the last term in the definition~\eqref{eq:Gdef} of $\mathcal G$, in the limit as $L\to\infty$  such a term becomes essential as the supports of the functions $\rho^\lambda_L$ become larger and larger and at the same time $L^1$-compactness of the sublevels of $\mathcal{G}$ is needed.  
	\end{remark}

\subsection{Limit PDEs}

Our first goal is to prove Theorem~\ref{thm:convlfixed}.

In order to do so, we need the following preliminary lemma.

\begin{lemma}
	Let $\rho_{0,L}\in L^1(\T_L)$ be such that
	\begin{equation}
	\label{eq:lblinfty}
	\inf_{\T_L}\rho_{0,L}\geq\eps_L>0
	\end{equation}
	and let $x_0(t)<\dots<x_{N-1}(t)$ be the deterministic particles which evolve according to~\eqref{eq:ode} on a time interval $[0,T]$ starting from $\rho_{0,L}$. Then, for all $k=0,\dots,N-1$ and for all $t\in[0,T]$
	\begin{equation}\label{eq:k+1k}
	x_{k+1}(t)-x_k(t)\leq\frac{1}{N}\Big(\frac{c_L}{\eps_L}+\frac{\|K'\|_{L^\infty}}{\|K''\|_{L^\infty}}\Big)e^{c_L\|K''\|_{L^\infty}t}.
	\end{equation}
\end{lemma}

\begin{proof} Recalling the ODE~\eqref{eq:ode}, one has that
	\begin{align*}
	\dot x_{k+1}(t)-\dot x_k(t)&=-\frac{c_L}{N}\Big[2K'(x_{k+1}-x_k)-\sum_{j\neq k,k+1}(K'(x_{k+1}-x_j)-K'(x_k-x_j))\Big]\notag\\
	&-\frac{N}{c_L}[\phi(\rho_{k+1})+\phi(\rho_{k-1})-2\phi(\rho_k)]\notag\\
	&\leq c_L\frac{\|K'\|_{L^\infty}}{N}+c_L\|K''\|_{L^\infty}(x_{k+1}(t)-x_k(t)) -\frac{N}{c_L}[\phi(\rho_{k+1})+\phi(\rho_{k-1})-2\phi(\rho_k)]
	\end{align*}
	Let now for $t\in[0,T]$ choose $k$ such that  $x_{k+1}(t)-x_k(t)=\max_j(x_{j+1}(t)-x_j(t))$. Then by monotonicity of $\phi$ and the definition of the deterministic particle approximation one has that $\phi(\rho_{k+1})+\phi(\rho_{k-1})-2\phi(\rho_k)\geq0$, hence
	\begin{equation*}
	\dot x_{k+1}(t)-\dot x_k(t)\leq c_L\frac{\|K'\|_{L^\infty}}{N}+c_L\|K''\|_{L^\infty}(x_{k+1}(t)-x_k(t)).
	\end{equation*}
	Moreover, one has that by assumption
	\begin{equation*}
	x_{k+1}(0)-x_k(0)\leq\frac{c_L}{N\eps_L}.
	\end{equation*} 
	Let us now consider the ODE
	\[
	\dot y(t)=\frac{c_L\|K'\|_{L^\infty}}{N}+c_L\|K''\|_{L^\infty}y(t),\qquad y(0)\leq \frac{c_L}{N\eps_L}.	
	\]
	One has that
	\[
	y(t)\leq\frac1N\Big(\frac{c_L}{\eps_L}+\frac{\|K'\|_{L^\infty}}{\|K''\|_{L^\infty}}\Big)e^{c_L\|K''\|_{L^\infty}t},
	\]
	hence~\eqref{eq:k+1k} is proved.

\end{proof}

\begin{proof}[Proof of Theorem~\ref{thm:convlfixed}:]
	 Let $\rho_{0,L}$ a bounded $L^1$ density strictly bounded from below on $\T_L$ by a constant $\eps_L$ as in the assumptions of the Theorem, and let $c_L=\int_{T_L}\rho_{0,L}\leq1$. W.l.o.g. we can assume that $\rho_{0,L}=(\rho_0)_L$ for some $\rho_0\in L^1(\R)\cap L^\infty(\R)$. In particular, if $N$ is sufficiently large, one has that
\begin{equation*}
\sup_{N}\|\rho^N_L(0)\|_{L^\infty(\T_L)}\leq C_1,\qquad\sup_N\Fcal_L(\rho^N_L(0))\leq C_0,
\end{equation*}
where $\rho^N_L$ is the deterministic particle approximation on $\T_L$ starting from $\rho_{0,L}$ (see~\eqref{eq:lambdacond}). In particular, Theorem~\ref{thm:compactness} guarantees that, up to subsequences, $\rho^N_L$ strongly converges in $L^1([0,T]\times \T_L)$ to a density $\rho_L\in L^1([0,T]\times\T_L)$.

Remember that, by~\eqref{eq:pdeapprox}, for every fixed $N \in \N, L>0$ the piecewise constant function $\rho^N_L(t,x)$ satisfies, for every $\varphi \in C^\infty_c((0,T)\times \T_L)$, the following equation	
	\begin{equation}\label{eq:pde3}
	\int_0^T \int_{\T_L} \rho^N_L(t,x) \partial_t \varphi(t,x) - \rho^N_L(t,x) \big[{K^\lin}'(\rho^N_L) \ast \rho^N_L(t,x) + \Phi(\rho^N_L(t,x)) \big]\partial_x \varphi(t,x) \dx\dt = 0
	\end{equation}
where we defined 
\begin{align*}
\Phi(\rho^N_{L}(t,x)) := \frac{N}{c_L} \sum_{k=0}^{N-1} \chi_{[x_k(t), x_{k+1}(t))} (x) \Bigl[ F_k(t) + \Bigl(\frac{N}{c_L}M_{\rho^N_L}(x) - k\Bigr)(F_{k+1}(t) - F_k(t)) \Bigr]
\end{align*}
and $F_k(t) = \phi (\rho_k(t)) - \phi(\rho_{k-1}(t))$. 

In the following we want to show that as $N\to+\infty$ the equation~\eqref{eq:pde3} gives the weak formulation of the PDE~\eqref{eq:mainPDE} for the density $\rho_L$.

In particular  we will prove  that 	as $N$ tends to $+\infty$
\begin{align}
&\int_0^T \int_{\T_L} \big(\rho^N_L(t,x) - \rho_L(t,x)\big) \partial_t \varphi(t,x) \dx \dt \longrightarrow 0,\label{eq:lim1}\\
&\int_0^T \int_{\T_L} \big((\rho^N_L(t,x){K^\lin}'(\rho^N_L) \ast \rho^N_L(t,x) - \rho_{L}K' \ast \rho_L (t,x) \big) \partial_x \varphi(t,x) \dx \dt \longrightarrow 0,\label{eq:lim2}
\end{align}
and, if as in the assumptions it holds that
\begin{equation*}
\rho_{0,L}\geq\eps_L>0
\end{equation*}
then
\begin{equation}
\lim_{N\to\infty}\int_0^T \int_{\T_L}\rho^N_L(t,x) \Phi(\rho^N_L(t,x)) \partial_x \varphi(t,x) \dx \dt ={- \int_0^T \int_{\T_L} \phi(\rho)\partial_{xx} \varphi(t,x)} \dx \dt.\label{eq:lim3}
\end{equation}

The convergence in~\eqref{eq:lim1} is an immediate consequence of the $L^\infty$ weak* compactness of the densities $\rho^N_L$ implied by Theorem~\ref{thm:existenceapprox}.

 Let us focus on the convergence in~\eqref{eq:lim2}. Simple computations lead to the equivalent expression
\begin{align*}
\int_0^T \int_{\T_L} &\big((\rho^N_L(t,x){K^\lin}'(\rho^N_L) \ast \rho^N_L(t,x) - \rho_LK'\ast \rho_L (t,x) \big) \partial_x \varphi(t,x) \dx \dt = \\
= &  \int_0^T \int_{\T_L} (\rho^N_L(t,x) - \rho_L(t,x)) K' \ast \rho_L(t,x) \partial_x \varphi(t,x) \dx \dt \\
& + \int_0^T \int_{\T_L} \rho^N_L(t,x) K'\ast\big( \rho^N_L(t,x) - \rho_L (t,x)\big) \partial_x \varphi(t,x) \dx \dt \\
& + \int_0^T \int_{\T_L} \rho^N_L(t,x) \big({K^\lin}'(\rho^N_L) - K' \big) \ast \rho^N_L(t,x)  \partial_x \varphi(t,x) \dx \dt.
\end{align*}
	The first and the second term of the r.h.s. of the above converge to $0$ as $N \to \infty$ because of the $L^\infty$ weak* compactness of the $\rho^N_L$.
 On the other hand, observe that
\begin{align*}
&\int_0^T \int_{\T_L}\rho^N_L(t,x) \big({K^\lin}'(\rho^N_L) - K' \big) \ast \rho^N_L(t,x)  \partial_x \varphi(t,x) \dx \dt  \\
=&\int_0^T\sum_{k,\,i=0}^{N-1} \int_{x_k}^{x_{k+1}} \rho_k(t) \rho_i(t) \Bigl[ -\int_{x_i}^{x_{i+1}} K'(x - y) \dy  + (x_{i+1} - x_i) \Big((1 - \chi_{\{0\}}(x_k - x_i))K'(x_k - x_i) \\
&+ Z_k(x) \big( (1 - \chi_{\{0\}}(x_{k+1} - x_i))K'(x_{k+1} - x_i)  -  (1 - \chi_{\{0\}}(x_k - x_i))K'(x_k - x_i) \big) \Big)\Bigr]\partial_x\varphi(x,t) \dx \dt, 
\end{align*}	  
	where we used $Z_k(x)$ to denote the function $\bigl(\frac{N}{c_L}M_{\rho^N_L}(x) - k\bigr)$ and we did not explicit the dependence from the time variable in the points $x_i(t), x_k(t)$ for simplicity of notation.

If $i \neq k, k+1$, recalling that $Z_k(x) \leq 1$, we have the following estimate 
\begin{align}\label{eq:ikkest}
 \int_{x_k}^{x_{k+1}} \rho_{k}(t)& \rho_{i}(t) (x_{i+1} - x_i) \Big| \frac{K(x - x_i) - K(x - x_{i+1})}{x_{i+1} - x_i} - K'(x_k - x_i) \notag\\
& - Z_k(x)(K'(x_{k+1} - x_i) - K'(x_k - x_i)) \Big| \dx \notag \\
 \leq & \frac{3c_L}{N} \int_{x_k}^{x_{k+1}} \rho_{k}(t) \|K'\|_{L^\infty}  \leq   \frac{3c_L^2\|K'\|_{L^\infty}}{N^2}.
\end{align}
If now $i = k+1$, we have that 
\begin{align}\label{eq:ik1est}
& \int_{x_k}^{x_{k+1}} \rho_{k}(t) \rho_{k+1}(t) (x_{k+2} - x_{k+1}) \Big| \frac{K(x - x_{k+1}) - K(x - x_{k+2})}{x_{k+2} - x_{k+1}} - K'(x_k - x_{k+1}) \notag\\
& \qquad \quad + Z_k(x)K'(x_k - x_{k+1})) \Big| \dx \leq \frac{3c_L^2\| K' \|_{L^\infty}}{N^2} .
\end{align}
Finally, for $i=k$ we get 
\begin{equation}\label{eq:ikest}
 \int_{x_k}^{x_{k+1}} \rho^2_{k}(t)  \left|\int_{x_k}^{x_{k+1}} K'(x - y) dy - (x_{k+1} - x_k) Z_k(x) K'(x_{k+1} - x_k)  \right| \dx \leq \frac{c_L^2(1+(x_{k+1}-x_k))}{N^2} \| K'\|_{L^\infty}, 
\end{equation}
and gathering together~\eqref{eq:ikkest},~\eqref{eq:ik1est} and~\eqref{eq:ikest} we conclude that 
\begin{align*}
\int_0^T \int_{\T_L}\rho^N_L(t,x) & \big({K^\lin}'(\rho^N_L) - K' \big) \ast \rho^N_L(t,x)  \partial_x \varphi(t,x) \dx\ dt 
\leq \frac{(7+L)\| \partial _x \varphi \|_{L^\infty}T\|K'\|_{L^\infty}}{N^2}  
\end{align*}
	where we have used that $\sum_k(x_{k+1}-x_k)=L$ and $c_L\leq 1$. Hence  the convergence claimed in~\eqref{eq:lim2} follows.
 
 	Let us now prove~\eqref{eq:lim3}.

 	 One has that 
	\begin{align*}
	&\int_0^T \int_{\T_L}\rho^N_L(t,x) \Phi(\rho^N_L(t,x)) \partial_x \varphi(t,x) \dx \dt= \notag\\
	&=\sum_{k=0}^{N-1} \int_{0}^T\int_{x_k}^{x_{k+1}}\frac{\phi(\rho^N_L(x_{k}))-\phi(\rho^N_L(x_{k-1}))}{x_{k+1}-x_k}\partial_x \varphi(t,x)\dx\dt\notag\\
	&+\sum_{k=0}^{N-1} \int_{0}^T\int_{x_k}^{x_{k+1}}\Big[\frac{\phi(\rho^N_L(x_{k+1}))-\phi(\rho^N_L(x_k))}{x_{k+1}-x_k}-\frac{\phi(\rho^N_L(x_{k}))-\phi(\rho^N_L(x_{k-1}))}{x_{k+1}-x_k}\Big]\Bigl(\frac{N}{c_L} M_{\rho^N_L}(x)-k\Bigr)\partial_x \varphi(t,x)\dx\dt\notag\\
	&=:I^N_1+I^N_2.
	\end{align*}

    We will treat the terms $I^N_1$ and $I^N_2$ separately. As for $I^N_1$, one has that integrating by parts it holds
    \begin{align*}
    I^N_1&=\sum_{k=0}^{N-1} \int_{0}^T\int_{x_k}^{x_{k+1}}\frac{\phi(\rho^N_L(x_{k}))-\phi(\rho^N_L(x_{k-1}))}{x_{k+1}-x_k}\partial_x \varphi(t,x)\dx\dt\notag\\
    &=-\sum_{k=0}^{N-1} \int_{0}^T\int_{x_k}^{x_{k+1}}\bigl[\phi(\rho^N_L(x_{k-1}))+\frac{x-x_k}{x_{k+1}-x_k}\bigl(\phi(\rho^N_L(x_k)))-\phi(\rho^N_L(x_{k-1}))\bigr)\bigr]\partial_{xx}\varphi(t,x)\dx\dt.
    \end{align*}
    Now notice that, since $\phi(\rho^N_L)$ converges in $L^1([0,T]\times\T_L)$ to $\phi(\rho_L)$ then
    \begin{equation*}
    \hat I^N_1:=-\sum_{k=0}^{N-1} \int_{0}^T\int_{x_k}^{x_{k+1}}\phi(\rho^N_L(x_k))\partial_{xx}\varphi(t,x)\dx\dt\quad\longrightarrow\quad-\int_{0}^\infty\int_{T_L}\phi(\rho_L)\partial_{xx}\varphi(t,x)\dx\dt.
    \end{equation*}
    Hence we want to show that $I^N_1-\hat I^N_1\to 0$ as $N\to\infty$. One has the following estimates: if $\mathrm{spt}(\varphi)\subset[0,T]\times\T_L$,
    \begin{align*}
    |I^N_1-\hat I^N_1|&\leq\|\varphi\|_{C^2}\sum_{k=0}^{N-1} \int_{0}^T\int_{x_k}^{x_{k+1}}\Big|\phi(\rho^N_L(x_{k-1}))+\frac{x-x_k}{x_{k+1}-x_k}\bigl(\phi(\rho^N_L(x_k)))-\phi(\rho^N_L(x_{k-1}))\bigr)-\phi(\rho^N_L(x_k))\Big|\dx\dt\notag\\
    &\leq2\|\varphi\|_{C^2}\int_0^T TV(\phi(\rho^N_L)(t))\max_{k}|x_{k+1}(t)-x_k(t)|\dt\notag\\
    &\leq\frac{2\|\varphi\|_{C^2}C(1+T)}{N}\Big(\frac{c_L}{\eps_L}+\frac{\|K'\|_{L^\infty}}{\|K''\|_{L^\infty}}\Big)e^{c_L\|K''\|_{L^\infty}T}
    \end{align*}
    where in the last inequality we have used~\eqref{eq:bvest} and~\eqref{eq:k+1k}. Thus we have the desired convergence.
    
    In order to prove~\eqref{eq:lim3} we are left to show that $I^N_2\to0$ as $N\to+\infty$. One has the following
    \begin{align*}
    I^N_2=\int_0^T\sum_{k=0}^{N-1}(\phi(\rho^N_L(x_{k+1}))-\phi(\rho^N_L(x_{k})))B_k(t)\dt +\int_0^T\sum_{k=0}^{N-1}(\phi(\rho^N_L(x_{k}))-\phi(\rho^N_L(x_{k-1})))B_k(t)\dt,
    \end{align*}
    with
    \[
    B_k(t)=\fint_{[x_k,x_{k+1})}\Big(\frac{N}{c_L}M_{\rho^N_L}(x)-k\Big)\partial_x\varphi(x,t)\dx.
    \]
    Notice then that $B_k=\bar B_k+\hat B_k$, where
    \begin{align*}
    \bar B_k(t)&=\fint_{[x_k,x_{k+1})}\Big(\frac{N}{c_L}M_{\rho^N_L}(x)-k\Big)\partial_x\varphi(x_k,t)\dx\notag\\
    &=\frac12\partial_{x}\varphi(x_k,t),\\
    \hat B_k(t)&=\fint_{[x_k,x_{k+1})}\Big(\frac{N}{c_L}M_{\rho^N_L}(x)-k\Big)(\partial_x\varphi(x,t)-\partial_x\varphi(x_k,t))\dx
    \end{align*}
    and 
    \begin{equation*}
    |\hat B_k(t)|\leq\frac12\|\varphi\|_{C^2}(x_{k+1}(t)-x_k(t)).
    \end{equation*}
    In particular,
    \begin{align}
    \int_0^T\sum_{k=0}^{N-1}(\phi(\rho^N_L(x_{k+1}))-\phi(\rho^N_L(x_{k})))\hat B_k(t)\dt&\leq\int_0^T TV(\phi(\rho^N_L)(t))\|\varphi\|_{C^2}\sup_k(x_{k+1}(t)-x_k(t))\dt\notag\\
    &\leq \frac{C(T,\|\varphi\|_{C^2}, \|K''\|_{L^\infty}, \|K'\|_{L^\infty}, \eps_L)}{N}\quad\longrightarrow\quad0,\label{eq:lastest}
    \end{align}
    where the last inequality follows from~\eqref{eq:bvest} and~\eqref{eq:k+1k}.
    
    On the other hand,
    
    \begin{align*}
    \Big|\int_0^T\sum_{k=0}^{N-1}(\phi(\rho^N_L(x_{k+1}))&+\phi(\rho^N_L(x_{k-1}))-2\phi(\rho^N_L(x_{k})))\bar B_k(t)\dt\Big|=\notag\\
    &=\frac12\Big|\int_0^T\sum_{k=0}^{N-1}\bigl[\phi(\rho^N_L(x_k))-\phi(\rho^N_L(x_{k-1}))\bigr]\bigl[\partial_x\varphi(x_{k-1},t)-\partial_x\varphi(x_{k},t)\bigr]\dt\Big|\notag\\&\leq\frac12\int_0^T TV(\phi(\rho^N_L)(t))\|\varphi\|_{C_2}\sup_k(x_{k+1}(t)-x_k(t))\dt
    \end{align*}
    which converges to $0$ as $N\to+\infty$ by the same estimate as in~\eqref{eq:lastest}.

\end{proof}	  

We now proceed to the proof of Theorem~\ref{thm:convinl}.

\begin{proof}[Proof of Theorem~\ref{thm:convinl}:] Let $\hat \rho_0>0$ a bounded $L^1$ density with unit mass. In particular, for every $L$ there exists $\eps_L>0$ such that
\begin{equation*}
(\hat\rho_0)_L(x)=\hat\rho_0(x)\chi_{[-L/2,L/2)}(x)\geq\eps_L,\qquad\forall\,x\in[-L/2,L/2).
\end{equation*} 
Denote now by ${(\rho_0)}_L$ the $L$-periodic extension of $(\hat\rho_0)_L$ or its corresponding function on $\T_L$. Denoting by $\rho^N_L$ the deterministic particle approximations defined starting from ${(\rho_0)}_L$, we proved in the previous theorem that they converge, up to subsequences, to a bounded $L^1$ solution $\rho_L$ of the following PDE in weak form
\begin{align}\label{eq:varphieq}
\int_0^T \int_{\T_L}  \rho_L(t,x) \partial_t \varphi(t,x) + \rho_{L}K' \ast \rho_L (t,x) \partial_x \varphi(t,x) -\phi(\rho_L)\partial_{xx}\varphi(t,x)\dx \dt=0,
\end{align} 
  where $\varphi\in C^\infty_c((0,T)\times\T_L)$. 
  
  Let now $\hat\psi\in C^\infty_c((0,T)\times\R)$. In particular, there exists $L$ such that 
  \[
  \mathrm{spt}\hat\psi\subset\!\subset [0,T]\times (-\bar L/2,\bar L/2).
  \] 
  For $L\geq\bar L$, denote by ${\psi}$ the $L$-periodic extension of $\hat\psi$. Then, by~\eqref{eq:varphieq} one has that
  \begin{align}\label{eq:psieq}
  \int_0^T \int_{\T_L}  \rho_L(t,x) \partial_t \psi(t,x) + \rho_{L}K' \ast \rho_L (t,x) \partial_x \psi(t,x) -\phi(\rho_L)\partial_{xx}\psi(t,x)\dx \dt=0,
  \end{align} 
  Denoting by ${\rho}_L$ the $L$- periodic function on $\R$ corresponding to $\rho_L$ and setting  $\hat\rho_L={\rho}_L\chi_{[-L/2,L/2)}$ and $\hat\psi={\psi}\chi_{[-L/2,L/2)}$,~\eqref{eq:psieq} rewrites as  
 \begin{align*}
 \int_0^T \int_{\R}  \hat\rho_L(t,x) \partial_t \hat\psi(t,x) + \hat\rho_{L}(t,x)K' \ast \rho_L (t,x) \partial_x \hat\psi(t,x) -\phi(\hat \rho_L)\partial_{xx}\hat\psi(t,x)\dx \dt=0,
 \end{align*}
 
 By Theorem~\ref{thm:compactness2} we know that, up to subsequences, $\hat\rho_L$ converge in $L^1([0,T]\times\R)$ to a bounded $L^1$ density $\rho$ as $L\to\infty$. In particular, as $L\to\infty$,
 
 \begin{equation*}
  \int_0^T \int_{\R}  \bigl(\hat\rho_L(t,x)-\rho(t,x)\bigr) \partial_t \hat\psi(t,x)-\bigl(\phi(\hat\rho_L)-\phi(\rho)\bigr)\partial_{xx}\hat\psi(t,x)\dx \dt\quad\longrightarrow\quad0.
 \end{equation*}
 
 and
 \begin{equation*}
  \int_0^T \int_{\R} \bigl(\hat\rho_{L}(t,x)K' \ast\hat \rho_L (t,x) -\rho(t,x)K'\ast\rho(t,x)\bigr)\partial_x \hat\psi(t,x)\dx\dt\quad\longrightarrow\quad0.
 \end{equation*}
 Thus we are left to show that
 \begin{equation*}
  \int_0^T \int_{\R} \hat\rho_{L}(t,x)K' \ast (\rho_L-\hat{\rho}_L) (t,x)\partial_x \hat\psi(t,x)\dx\dt\quad\longrightarrow\quad0.
 \end{equation*}
 Using the fact that $\|\rho_L\|_{\infty}\leq\gamma_1+\gamma_2 T$ (see Theorem~\ref{thm:existenceapprox}) and  that $K'\in L^1(\R)\cap L^\infty(\R)$ (see~\eqref{eq:k4}), one has that
 \begin{align*}
 \int_0^T \int_{\R} &|\hat\rho_{L}(t,x)|\int_{\R}|K'(x-y)| |\rho_L(t,y)-\hat{\rho}_L(t,y)| \dy|\partial_x \hat\psi(t,x)|\dx\dt\leq\notag\\
 &= \int_0^T \int_{\R} |\hat\rho_{L}(t,x)|\int_{\R\setminus[-L/2,L/2]}|K'(x-y)| |\rho_L(t,y)|\dy|\partial_x \hat\psi(t,x)|\dx\dt\notag\\
 &=J^L_1+J^L_2,
 \end{align*}
 with
 \begin{align*}
 J^L_1&=\int_0^T \int_{[-L/2,L/2]} |\hat\rho_{L}(t,x)|\int_{\R\setminus[-3L/4,3L/4]}|K'(x-y)| |\rho_L(t,y)|\dy|\partial_x \hat\psi(t,x)|\dx\dt,\notag\\
 J^L_2&=\int_0^T \int_{[-L/2,L/2]} |\hat\rho_{L}(t,x)|\int_{[-3L/4,-L/2]\cup[L/2,3L/4]}|K'(x-y)| |\rho_L(t,y)|\dy|\partial_x \hat\psi(t,x)|\dx\dt.
 \end{align*}
 On one hand,
 \begin{equation*}
 J^L_1\leq \|\varphi\|_{C^1}c_L2R\int_{\{|z|>L/4\}}|K'(z)|\dz\quad\longrightarrow\quad0.
 \end{equation*}
 On the other hand,
 \begin{equation*}
 J^L_2\leq \|\varphi\|_{C^1}c_L\|K'\|_{L^\infty}\int_{\{L/4<|y|<L/2\}}\hat \rho_L(y)\dy\quad\longrightarrow\quad0,
 \end{equation*}
 where we used the tightness of the measures $\hat{\rho}_L$ proved in Remark~\ref{rem:lambda}.
 
 Thus the proof of Theorem~\ref{thm:convinl} is concluded.
  
\end{proof}

Let us move to the proof of the last main result of this paper, namely Theorem~\ref{thm:convfinal}.

\begin{proof}
	[Proof of Theorem~\ref{thm:convfinal}: ] Let $\rho_0$ and $\{\rho_0^\lambda\}_{\lambda\in\N}$ as in the statement of the theorem. In particular, if $\lambda$ is sufficiently large we can assume that 
	\begin{align*}
	\sup_{N,L,\lambda}\bigl\|\bigl(\rho^\lambda_0\bigr)^N_L(0)\bigr\|_{L^\infty}\leq C_1,\qquad \sup_{N,L,\lambda}\Fcal_L\bigl(\bigl(\rho^\lambda_0\bigr)^N_L(0)\bigr)\leq C_0.
	\end{align*} 
	By the previous theorem, we have as $N\to+\infty$ and then $L\to+\infty$ in the deterministic particle approximations $(\rho^\lambda)^N_L$ limit densities $\rho^\lambda$ in $L^1(\R)\cap L^\infty(\R)$ which satisfy the PDE in weak form
	\begin{equation}\label{eq:weakfinal}
	\int_0^T\int_{\R}\rho^\lambda\partial_t\varphi+\rho^\lambda K'\ast\rho^\lambda\partial_x\varphi-\phi(\rho^\lambda)\partial_{xx}\varphi\dx\dt=0.
	\end{equation}
	By the second part of Theorem~\ref{thm:compactness2}, the sequence $\{\rho^\lambda\}_{\lambda\in\N}\subset L^1\cap L^\infty([0,T]\times\R)$ admits a (not relabeled) converging subsequence to some density $\rho$ with initial datum $\rho_0$. Since the weak formulation~\eqref{eq:weakfinal} is continuous with respect to strong $L^1$ limits, $\rho$ satisfies as well the PDE~\eqref{eq:mainPDE} in the weak sense and therefore  the theorem is proved. 
\end{proof}

\end{document}